\documentclass[11pt]
{amsart}
\usepackage{amssymb,amsmath,amsthm,amsfonts,amsopn,url,xcolor,hyperref,enumerate,mathtools,microtype,MnSymbol,comment}
\usepackage[normalem]{ulem}
\usepackage[all]{xy}
\input xy
\xyoption{all}
\usepackage{amscd}
\usepackage{soul}

\theoremstyle{plain}
\newtheorem{thm}{Theorem}[section]

\newtheorem{prop}[thm]{Proposition}
\newtheorem{lemma}[thm]{Lemma}
\newtheorem{cor}[thm]{Corollary}

\theoremstyle{definition}
\newtheorem{defn}[thm]{Definition}
\newtheorem*{defn*}{Definition}
\newtheorem*{question*}{Question}

\newtheorem{example}[thm]{Example}
\newtheorem*{example*}{Example}
\newtheorem{rem}[thm]{Remark}
\newtheorem*{rem*}{Remark}
\newtheorem{notation}[thm]{Notation}

\newcommand{\ideal}[1]{\mathfrak{#1}}
\newcommand{\m}{\ideal{m}}

\newcommand{\p}{\ideal{p}}

\newcommand{\func}[1]{\mathrm{#1} \,}

\newcommand{\depth}{\func{depth}}

\newcommand{\coker}{\func{coker}}
\newcommand{\im}{\func{im}}

\newcommand{\ra}{\rightarrow}

\DeclareMathOperator{\len}{\lambda}

\DeclareMathOperator{\ann}{ann}

\DeclareMathOperator{\Hom}{Hom}

\newcommand{\be}{\begin{enumerate}}
\newcommand{\ee}{\end{enumerate}}

\newcommand{\li}%{{\mathrm{lic}}}
 {\leftfootline}

\newcommand{\onto}{\twoheadrightarrow}

\newcommand{\into}{\hookrightarrow}

\newcommand{\cM}{\mathcal{M}}
\newcommand{\cP}{\mathcal{P}}
\renewcommand{\phi}{\varphi}

\DeclareMathOperator{\Soc}{Soc}

\newcommand{\fp}{\mathfrak{p}}

\newcommand{\subsel}{submodule selector}

\newcommand{\dual}{\smallsmile}
\newcommand{\cl}{{\mathrm{cl}}}
\let\int\relax
\DeclareMathOperator{\int}{i}

\newcommand{\cvl}{\color{violet}}

\newcommand{\fg}{finitely generated}

\DeclareMathOperator{\soc}{soc}

\DeclareMathOperator{\core}{-core}
\DeclareMathOperator{\hull}{-hull}

\newcommand{\fcog}{finitely cogenerated}
\newcommand{\bemp}{basically empty}
\newcommand{\bfull}{basically full}
\newcommand{\clr}{complete local ring}
\newcommand{\RaR}{{\mathrm {RR}}}
\newcommand{\Jcolsym}[1]{{#1} {\rm bf}}

\newcommand{\Jcol}[3]{{#2}^{\Jcolsym #1}_{#3}}

\newcommand{\Jintrelsym}[1]{{#1}{\rm be}}

\newcommand{\Jintrel}[3]{{#3}_{\Jintrelsym{#1}}^{#2}}

\author{Neil Epstein}
\address{Department of Mathematical Sciences \\ George Mason University \\ Fairfax, VA  22030}
\email{nepstei2@gmu.edu}

\author{Rebecca R.G.}
\address{Department of Mathematical Sciences \\ George Mason University \\ Fairfax, VA  22030}
\email{rrebhuhn@gmu.edu}

\author{Janet Vassilev}
\address{Department of Mathematics and Statistics \\ University of New Mexico \\ Albuquerque, NM 87131}
\email{jvassil@math.unm.edu }

\title[Integral closure, basically full closure, and duality]{Integral closure, basically full closure, and duals of nonresidual closure operations}

\date{\today}
\begin{document}
\begin{abstract}
We develop a duality for operations on nested pairs of modules that generalizes the duality between absolute interior operations and residual closure operations from \cite{nmeRG-cidual}, extending our previous results to the expanded context.  We apply this duality in particular to integral 
and basically full closures and their respective cores to obtain integral
and basically empty interiors and their respective hulls.   
We also dualize some of the known formulas for the core of an ideal to obtain formulas for the hull of a submodule of the injective hull of the residue field.  The article concludes with illustrative examples in a numerical semigroup ring.
\end{abstract}

\maketitle
\setcounter{tocdepth}{1} 
\tableofcontents

\section{Introduction}

Integral closure of ideals and modules is of central importance in commutative algebra, and thus has been extensively studied (e.g.,  \cite{HuSw-book,Vascon05} are books on the subject). In fact, through his examination of integral closure of fractional ideals of a domain, Krull \cite{Krull-IT} commenced the first study of properties of closure operations in commutative algebra.  Integral closure is especially important because of its relation to other notions in commutative ring theory.

An ideal $J\subseteq I$, was originally termed a reduction of $I$ by Northcott and Rees in \cite{NR-idealreductions} if there exists a natural number $r$, such that $I^{n+1}=JI^n$ for all $n \geq r$; they showed that when $R$ is Noetherian, $J$ is a reduction of $I$ if and only if the two ideals have the same integral closure.  Rees and Sally's original definition of the core of an ideal  \cite{RS-core} was through reductions; hence, intrinsically relies on the notion of integral closure. 

Basically full closure with respect to an ideal $J$ was first defined for $J=\m$ in \cite{HRR-bf}, and further developed and expanded in \cite{HLNR, Va-*full,VaVr, Rush-contracted,Dao-colon}. If $R$ is a Noetherian domain, Ratliff and Rush showed in \cite[Proposition 3.1]{RRdeltamod} that an ideal is $J$-basically full (closed) for all ideals $J$ if and only if the ideal is integrally closed  (although they didn't phrase it that way; see also Vasconcelos \cite[Proposition 1.58]{Vascon05} for a more on-point version). 

Both integral closure and basically full closure are examples of closure operations that are nonresidual, i.e., it is possible to come up with surjective $R$-module maps $\pi:M \twoheadrightarrow P$ such that the closure of $\ker(\pi)$ in $M$ is not equal to the pre-image of the closure of 0 in $P$.

In previous work \cite{nmeRG-cidual, ERGV-chdual}, the authors explored a duality operation that takes a residual closure operation on modules over a complete local ring (such as tight closure, module closures, or Frobenius closure) to an absolute interior operation (e.g., in the above cases, the tight interior explored in \cite{nmeSc-tint}, trace as in \cite{PeRG}, and the novel Frobenius interior, respectively).  We showed that the interior of the ring tends to coincide with the common annihilator of closures of submodules (generalizing \emph{test ideals}). We also dualized the notions of reductions and core of a submodule with respect to a residual closure operation to obtain \emph{expansions} and the \emph{hull} of a submodule with respect to an absolute interior operation.

However, the above work depends on the closure operation in question being \emph{residual}, even though some closure operations of interest are not residual. Motivated by the non-residual examples of integral closure and basically full closure, we extend our results (and our duality operation) to nonresidual closure operations (Section \ref{sec:nonresidualdual}). 
We find that the dual of such a closure operation is not an absolute interior operation, but rather a \emph{relative} one.  That is, if $L \subseteq N \subseteq M$, the interior of $L$ in $N$ may be strictly smaller than the interior of $L$ in $M$.  Even so, we are able to extend our duality results involving cores and hulls (Section \ref{sec:nonresidualcorehullduality}).  

To develop our newly extended duality operation we define a \emph{pair operation} (Section \ref{sec:backgroundpairops}), which assigns to a pair $L \subseteq M$ of modules a submodule of $M$ which does not have any predetermined containment relation with respect to $L$.  In the case of interiors (resp., closures), it is contained in (resp., contains) $L$.

 In \cite{HRR-bf}, Heinzer, Ratliff and Rush define a submodule $L$ of a finitely generated module $M$ to be basically full if for every submodule $N \subseteq M$ properly containing $L$, no minimal generating set of $L$ can be extended to a generating set of $N$.  We define a basically empty submodule $A$ of an Artinian module $B$ as one satisfying the property that for all proper submodules $C$ in $A$ no minimal cogenerating set (see \cite{Vam} or Section \ref{sec:cogen}) of $B/A$ can be extended to a minimal cogenerating set of $B/C$.  We spend a considerable amount of effort analyzing the operations of  ($J$-)\emph{basically full closure} and its dual, the ($J$-)\emph{basically empty interior}, for an ideal $J$ of $R$ (Section \ref{sec:bfcbei}). Both operations have nice expressions in terms of colons:  \[\Jcol{J}{N}{M}=(J N:_M J) \text{ and } \Jintrel{J}{M}{N}=J(N:_MJ)\] for $N \subseteq M$ $R$-modules.  We show in Theorem \ref{thm:Jcoldual} that these operations are dual to each other. Not only are the $J$-basically full closure and $J$-basically empty interiors dual to each other, but we show in Lemma~\ref{lem:bfbe} that basically full submodules of a finitely generated module are also the dual notion to basically empty submodules of an Artinian module.

We then apply our duality results to integral closure and integral interior (Section \ref{sec:integralhulls}). 
We determine formulas for the integral hull of some submodules of the injective hull of the residue field in Theorem~\ref{thm:hullform} and Theorem~\ref{thm:hullIbf}.
We additionally realize formulas for  $*\core(I)$  in terms of the $I$-basically full interior of a minimal reduction of $I$ (Corollary \ref{cor:*corebe}) and $*\hull(0:_E I)$ in terms of the  $I$-basically full closure of $(0:_E J)$ for $J$ a minimal reduction of $I$ (Theorem~\ref{thm:*hullform}). 
We conclude the article by computing examples of basically full closures and basically empty interiors of ideals and their cores and hulls respectively in the completion of a numerical semigroup ring (Section \ref{sec:examples}).

\section{Background and pair operations}
\label{sec:backgroundpairops}

In this section, we define pair operations and several important properties. We describe the duality between closure operations and interior operations over a \clr\ as first given by the first two named authors in \cite{nmeRG-cidual}. We then recall the definition of a Nakayama closure, $\cl$-reductions, and the $\cl$-core, and give some of their properties.

We had previously (see \cite{nmeRG-cidual}) defined a duality between submodule selectors which we parlayed into a duality between residual closure operations and (absolute) interior operations.  However, in this paper we extend this duality to the more general context of pair operations, allowing us to deal effectively with nonresidual closure operations such as integral closure and basically full closure.

\begin{defn}[{\cite{nmeRG-cidual}}]
Let $R$ be a ring, not necessarily commutative. Let $\cM$ be a class of (left) $R$-modules that is closed under taking submodules and quotient modules.  Let $\cP := \cP_\cM$ denote the set of all pairs $(L,M)$ where $M \in \cM$ and $L$ is a submodule of $M$ in $\cM$.

A \emph{\subsel} is a function $\alpha: \cM \rightarrow \cM$ such that \begin{itemize}
 \item $\alpha(M) \subseteq M$ for each $M \in \cM$, and
 \item for any isomorphic pair of modules $M, N \in \cM$ and any isomorphism $\phi: M \rightarrow N$, we have $\phi(\alpha(M)) = \alpha(\phi(M))$.
  \end{itemize}
  
An \emph{absolute} interior operation \footnote{ In \cite{nmeRG-cidual}, this is simply called an interior operation.} is a submodule selector that is
\begin{itemize}
    \item \emph{order-preserving}, i.e. for any $L \subseteq M \in \cM$, $\alpha(L) \subseteq \alpha(M)$, and
    \item  \emph{idempotent}, i.e. for all $M \in \cM$, $\alpha(\alpha(M))=\alpha(M)$.
\end{itemize} 

A submodule selector $\alpha$ is \emph{functorial} if for any $g: M \rightarrow N$ in $\cM$, we have $g(\alpha(M)) \subseteq \alpha(N)$.
\end{defn}

Next we define pair operations and an assortment of properties that they can have. Note that both closure operations and relative interior operations are examples of pair operations.  The notion of a pair operation is useful in that it encompasses various sorts of operations  (closures, interiors, hulls, cores) within a common framework that we can manipulate in a uniform way.  Moreover, it allows us to analyze our duality operations in a symmetric way, and it makes proofs shorter.
 
As such, we have the following:
\begin{defn}\label{def:pairops}
Let $\cP$ be a collection of pairs $(L,M)$, where $L$ is a submodule of $M$, 
such that whenever $\phi:M \to M'$ is an isomorphism and $(L,M) \in \cP$, $(\phi(L),M') \in \cP$ as well.

A \emph{pair operation} is a function $p$ that sends each pair $(L,M) \in \cP$ to a submodule $p(L,M)$ of $M$, in such a way that whenever $\phi: M \ra M'$ is an $R$-module isomorphism
 and $(L,M) \in \cP$, then 
  $\phi(p(L,M)) = p(\phi(L),M')$.  When $(L,M) \in \cP$, we say that $p$ is \begin{itemize}
    \item \emph{idempotent} if whenever $(L,M) \in \cP$ and $(p(L,M), M) \in \cP$, we always have $p(p(L,M),M)=p(L,M)$;
    \item \emph{extensive}
    if we always have $L \subseteq p(L,M)$;
    \item \emph{intensive} if we always have $p(L,M) \subseteq L$;
    \item \emph{order-preserving on submodules} if whenever $L \subseteq N \subseteq M$ such that $(L,M), (N,M) \in \cP$, we have $p(L,M) \subseteq p(N,M)$;
    \item \emph{order-preserving on ambient modules} if whenever  $L \subseteq N \subseteq M$
    such that $(L,N), (L,M) \in \cP$, 
    we have $p(L,N) \subseteq p(L,M)$;
    \item \emph{surjection-functorial} if whenever $\pi:M \twoheadrightarrow M'$ is a surjection and $(L,M),(\pi(L),M') \in \cP$, we have $\pi(p(L,M)) \subseteq p(\pi(L),M')$.  Equivalently,  when $(L,M) \in \cP$ and for 
    $U \subseteq M$, $((L+U)/U,M/U) \in \cP$, then $(p(L,M)+U)/U \subseteq p((L+U)/U, M/U)$;
    \item \emph{functorial} if whenever $g: M \ra M'$ and $(L,M),(g(L),M') \in \cP$, we have $g(p(L,M)) \subseteq p(g(L), M')$. (Note that if $(g(L),g(M))$ is also in $\cP$, it is equivalent that $p$ be both  order-preserving on ambient modules and surjection-functorial, by the usual epi-monic factorization); 
    \item a \emph{closure operation} if it is extensive, order-preserving on submodules, and idempotent;
    \item a \emph{(relative) interior operation} if it is intensive, order-preserving on submodules, and idempotent;
     \item \emph{absolute} if whenever $L \subseteq N\subseteq M$ 
     such that $(L,M), (L,N) \in \cP$, 
     we have $p(L,M) = p(L,N)$;
    \item \emph{residual} if whenever $L \subseteq N \subseteq M$ 
    are such that $(N/L, M/L), (N,M) \in \cP$, 
    we have 
   $p(N,M) = \pi^{-1}(p(N/L, M/L))$, where $\pi: M \onto M/L$ is the natural surjection;
    \item \emph{restrictable} if whenever $L \subseteq M$ and $N \subseteq M$ are such that $(L\cap N, N), (L,M) \in \cP$, we have $p(L\cap N, N) \subseteq p(L,M)$, or equivalently $p(L\cap N, N) \subseteq p(L,M) \cap N$.
\end{itemize}
\end{defn}

\begin{notation}
Throughout the paper, when $p$ is an extensive (e.g., closure) operation, we will write $N_M^p$ for $p(N,M)$, and when $p$ is an intensive (e.g., relative interior) operation, we will write $N_p^M$ for $p(N,M)$.
\end{notation}

\begin{rem}
Recall that in \cite{nmeRG-cidual}, an extensive operation $e$ was called \emph{residual} if whenever $q: M \onto P$ were a surjection, we had $(\ker q)^e_M = q^{-1}(0^e_P)$.  Note that this is equivalent to the definition given above.  Indeed it follows from the first isomorphism theorem that the current definition implies the older one.  For the converse, suppose $e$ is a residual  operation in the sense of \cite{nmeRG-cidual} and let $L \subseteq N \subseteq M$ such that  
$(N,M)$, $(0, M/N)$, $(N/L, M/L) \in \cP$.  Let $\pi: M \onto M/L$ and $q: M/L \onto M/N$ be the natural surjections. Then \[
    N^e_M = (\ker (q \circ \pi))^e_M = (q \circ \pi)^{-1}(0^e_{M/N}) 
    = \pi^{-1}(q^{-1}(0^e_{M/N})) = \pi^{-1}((N/L)^e_{M/L}).
\]
\end{rem}
It is also worth noting that any residual pair operation is already extensive, since in the special case $N=L$ of the definition, we have \[L = \pi^{-1}(0)  \subseteq \pi^{-1}(p(L/L, M/L)) = p(L,M).\]

We note the following lemma relating restrictability to order-preservation.

\begin{lemma}\label{lem:restrict}
Let $p$ be a pair operation on $\cP$.  \begin{enumerate}
    \item If $p$ is restrictable, then it is order-preserving on ambient modules.
    \item If $p$ is order-preserving on both submodules and ambient modules and whenever $\cP$ contains $(L,N)$ and $(N,M)$, it also contains $(L,M)$, then $p$ is restrictable.
\end{enumerate}
Hence, if $p$ is order-preserving on submodules (e.g. a closure or interior operation) and whenever $\cP$ contains $(L,N)$ and $(N,M)$, it also contains $(L,M)$, then restrictability is equivalent to order-preservation on ambient modules.
\end{lemma}

\begin{proof}
First suppose $p$ is restrictable.  Let $L \subseteq N \subseteq M$ such that $(L,N)$ and  $(L,M)$ are in $\cP$.  Then $p(L,N) = p(L \cap N, N) \subseteq p(L,M)$.

Conversely, suppose $p$ is order-preserving on both submodules and ambient modules.  Let $L,N$ be submodules of $M$ such that $(L\cap N, N)$ and $(L,M)$ are in  $\cP$.  By our hypothesis on $\cP$, $(L \cap N, M) \in \cP$. Then 
\begin{align*}
p(L \cap N, N) &\subseteq p(L\cap N, M) \text{ (by order-preservation on ambient modules)}\\ 
&\subseteq p(L,M) \text{ (by order-preservation on submodules).}
\end{align*}
\end{proof}

The following example illustrates why the converse of Lemma~\ref{lem:restrict}{\it (1)} does not hold:
\begin{example}\label{ex:nonrestricablepair}
Let $R=k[[x,y]]$.  Recall the Ratliff-Rush operation on the ideals of $R$:
\[
I_R^\RaR:=\bigcup\limits_{n=1}^{\infty} (I^{n+1}:_R I^n)
\]
first considered by Ratliff and Rush in \cite{RRoperationintro}.
Many, including Ratliff and Rush \cite{RRdelta}, have referred to this operation as the Ratliff-Rush closure; however, we point out that it is not order-preserving (on submodules) and thus not truly a closure operation.

Let $\cM=\{I \subseteq R\mid I \text{ an ideal}\}$ and $\cP=\{(I,J) \mid I \subseteq J \text{ ideals in } \cM\}$.
We define a pair operation $\alpha$ on $\cP$ via $\alpha(I,J)=I^{\RaR}_R\cap J$.  
Note that we also have
\[\alpha(I,J)=I_J^{\RaR}= \bigcup\limits_{n=1}^{\infty} (I^{n+1}:_J I^n).\]

Let $I=(x^3,y^3)$ and  $J=(x^4,xy,y^4)$ and note that $I \cap J=(x^4,x^3y,xy^3,y^4)$.
Following work of Ratliff and Rush \cite[Section 2]{RRdelta}, we see that
\[ (I \cap J)^{\RaR}_R = (I \cap J)_J^{\RaR}=(x^4,x^3y,x^2y^2,xy^3,y^4), \quad I_R^{\RaR}=I,\] 
\[\text{and } (x^4,x^3y,x^2y^2,xy^3,y^4) \nsubseteq I.\]
Hence $\alpha$ is a pair operation that is order-preserving on ambient modules (by definition), but neither restrictable nor order-preserving on submodules.
\end{example}

In previous papers \cite{nmeRG-cidual,ERGV-chdual}, the authors primarily studied residual closure operations. In this paper, we extend our work to nonresidual closures, necessitating some changes and additions to the definitions of those papers.
In extending the duality of \cite{nmeRG-cidual} to nonresidual closures, we will see that the appropriate dual of a not necessarily residual closure is a relative interior as defined in Definition~\ref{def:pairops}.

We recall the definition of the dual of a residual closure operation from \cite{nmeRG-cidual} to give context to our definition of a more general pair duality in Section \ref{sec:nonresidualdual}.

For the following several definitions and results, $R$ is a complete Noetherian local ring with maximal ideal $\m$, residue field $k$, and $E := E_R(k)$ the injective hull. We will use $^\vee$ to denote the Matlis duality operation, $\Hom_R(\_,E)$. $\cM$ is a category of $R$-modules closed under  taking submodules and quotient modules, and such that for all $M \in \cM$, $M^{\vee\vee} \cong M$. 
For example, $\cM$ could be the category of \fg\ $R$-modules, or of Artinian $R$-modules.

\begin{defn}[{\cite[Definition 3.1]{nmeRG-cidual}}]\label{def:oldduality}
Let $R$ be a complete Noetherian local ring.
Let $S(\cM)$ denote the set of all submodule selectors on $\cM$.  Define $\dual: S(\cM) \rightarrow S(\cM^\vee)$ as follows: For $\alpha \in S(\cM)$ and $M\in \cM^\vee$, \[
\alpha^\dual(M) := (M^\vee / \alpha(M^\vee))^\vee,
\]
considered as a submodule of $M$ in the usual way.  
\end{defn}

\begin{thm}
\label{thm:smsdual}
Let $R$ be a complete Noetherian local ring and
 $\alpha$ a submodule selector on $\cM$.  Then: \begin{enumerate}
 \item\label{it:duality} \cite[Theorem 3.2]{nmeRG-cidual} $(\alpha^\dual)^\dual = \alpha$,
 \item \cite[Theorem 4.3]{nmeRG-cidual} If $\alpha(M):=0_M^\cl$ for a residual closure operation $\cl$, then $\alpha^\dual$ is an interior operation. Conversely, if $\alpha$ is an interior operation, then $\alpha^\dual(M)=0_M^\cl$ for a uniquely determined residual closure operation $\cl$.
\end{enumerate}
\end{thm}

\begin{rem}
In consequence, if $\cl$ is a residual closure operation on $\cM$, its dual interior operation can be expressed as
\[\int(M)=\left(\frac{M^\vee}{0_{M^\vee}^{\cl}}\right)^\vee,\]
where $M \in \cM^\vee$.

 We note for clarity that in the terminology of the current paper, the interior operations given in Theorem~\ref{thm:smsdual} above are \emph{absolute}.
\end{rem}

Next we recall the definitions of the $\cl$-core and $\int$-hull, which we will study for several closure and interior operations.

\begin{defn}[compare {\cite[Section 2]{ERGV-chdual} or \cite{FoVa-core}}]
Let $R$ be an associative ring, not necessarily commutative, and $\cl$ a closure operation defined on a class $\cP$ of pairs of (left) $R$-modules. 
Let $(N,M) \in \cP$.  We say that $L \subseteq N$ is a \emph{$\cl$-reduction of $N$ in $M$} if $(L,M) \in \cP$ and $L^{\cl}_M=N^{\cl}_M$.

Note that if $(L,M),(N,M) \in \cP$, $L \subseteq N \subseteq L_M^\cl$ if and only if $L$ is a $\cl$-reduction of $N$ in $M$.

 If $(N,M) \in \cP$, the \emph{$\cl$-core of $N$ with respect to $M$} is the intersection of all $\cl$-reductions of $N$ in $M$, or
\[\cl\core_M(N):= 
\bigcap \{L \mid L \subseteq N \subseteq L_M^{\cl} \text{ and } (L,M) \in \cP\} .\]
As convention, when taking the $\cl$-core of an ideal $I$ in $R$, we will denote $\cl\core_R(I)$ by $\cl\core(I)$.
\end{defn}

\begin{defn}[{\cite[Definition 1.2]{nme-sp}}]\label{def:Nakcl}
Let $(R,\m)$ be a Noetherian local ring and $\cl$ be a closure operation on the class of pairs $\cP$ of \fg\ $R$-modules.  We say that $\cl$ is a \emph{Nakayama closure} if for $L \subseteq N \subseteq M$ \fg\ $R$-modules, if $L\subseteq N \subseteq (L+\m N)^{\cl}_M$ then $L^{\cl}_M=N^{\cl}_M$.
\end{defn}

Nakayama closures are the class of closures where we know that minimal $\cl$-reductions exist. The result below was stated for ideals in \cite[Lemma 2.2]{nme*spread}, and \cite[Section 1]{nme-sp} clarified that the result holds for modules. We repeat the version stated as \cite[Proposition 2.9]{ERGV-chdual}:

\begin{prop}\label{NR Ext}
Let $(R, \m)$ be a Noetherian local ring and $\cl$ a Nakayama closure operation on the class of \fg\  modules, or a closure operation on the class of  Artinian $R$-modules $\cM$. Let $N \subseteq M$ be elements of $\cM$. For any $\cl$-reduction $L$ of $N$ in $M$, there exists a minimal $\cl$-reduction $K$ of $N$ in $M$ such that $K \subseteq L$. Moreover, in the former case,  any minimal generating set of $K$ extends to one of $L$.
\end{prop}

The dual notion to the $\cl$-core is the $\int$-hull, here only defined for absolute interiors. See Section~\ref{sec:nonresidualcorehullduality}  for a generalization to relative interiors.

\begin{defn}[{\cite[Section 6]{ERGV-chdual}}]\label{def:absintexp}
Let $R$ be an associative (not necessarily commutative) ring and $A \subseteq B$ (left) $R$-modules.  Let $\int$ be an absolute interior operation on a class of (left) $R$-modules $\cM$ that operates on at least the submodules of $B$ that contain $A$.  
We say $C$ with $A\subseteq C \subseteq B$ is an \emph{$\int$-expansion of $A$ in $B$} if $\int(A)=\int(C)$.

Let $R$ be an associative (not necessarily commutative) ring and $\int$ an absolute interior operation defined on a class of (left) $R$-modules $\cM$ closed under taking submodules. 

If $A \subseteq B$ are elements of $\cM$, the \emph{$\int$-hull of a submodule $A$ with respect to $B$} is the sum of all $\int$-expansions of $A$ in $B$, or
 \[
\int\hull^B(A):=\sum_{\int(C) \subseteq A \subseteq C \subseteq B} C.
\]
\end{defn}

In this paper, we will deal with relative interior operations. The above definitions are the same for relative interiors, with $\int(A)$ replaced with $\int(A,B)$ if $B \supseteq A$ is the ambient module. We will discuss a relative version of a Nakayama interior in Section \ref{sec:nonresidualdual}, and relative versions of cores and hulls in Section \ref{sec:nonresidualcorehullduality}.

\section{Dual of a nonresidual closure operation}
\label{sec:nonresidualdual}

In this section we extend the work of \cite{nmeRG-cidual} to closure operations that are not necessarily residual, so that we can apply the results of that paper to integral closure and basically full closure. We describe the closure-interior duality in this case and prove that the dual of a Nakayama closure is a Nakayama interior, with a new relative version of the definition of a Nakayama interior.

We start by defining the dual of an arbitrary pair operation.

\begin{defn}
\label{def:nonresidualdual}
Let $R$ be a complete local ring. Let $p$ be a pair operation on a class of pairs of Matlis-dualizable $R$-modules $\cP$ as in Definition \ref{def:pairops}. Set $\cP^\vee := \{(A,B) \mid ((B/A)^\vee, B^\vee) \in \cP\}$, and define the dual of $p$ by 
\[
p^\dual(A,B):=\left(\frac{B^\vee}{p\left(\left(\frac{B}{A}\right)^\vee,B^\vee\right) } \right)^\vee.
\]
It's an easy exercise to show that $(\cP^\vee)^\vee=\cP$.
\end{defn}

We show that when applied to a residual closure operation, this definition agrees with the definition of closure-interior duality from \cite[Definition 3.1]{nmeRG-cidual}:

\begin{prop}
\label{pr:residualcasedualthesame}
Let $R$ be a complete local ring. Let $\cl$ be a residual closure operation on pairs of modules $L \subseteq M$ in a class $\cM$ that is closed under taking submodules and quotient modules and set $\alpha(M):=0_M^\cl$ for any $M \in \cM$. Then for any $A \subseteq B$ in $\cM^\vee$, $\alpha^\dual(A)=A^B_{\cl^\dual}$, 
where $\alpha^\dual$ is defined as in Definition~\ref{def:oldduality} and $\cl^\dual$ is as in Definition~\ref{def:nonresidualdual}.
\end{prop}

\begin{proof}
By definition,
\[\alpha^\dual(A)=\left(\frac{A^\vee}{0_{A^\vee}^\cl} \right)^\vee.\]
Let $\pi:B^\vee \to A^\vee$ be the natural surjection. Then
\[0_{A^\vee}^\cl=\pi\left(\left((B/A)^\vee\right)_{B^\vee}^\cl \right).\]
Set 
\[C=\left((B/A)^\vee \right)_{B^\vee}^\cl.\]
Notice that $\ker(\pi)=(B/A)^\vee \subseteq C$.
Hence
\[\bar{\pi}:B^\vee/C \to A^\vee/\pi(C)\]
is an isomorphism. 
This implies that
\[\alpha^\dual(A)=\left(A^\vee/\pi(C) \right)^\vee \cong \left(B^\vee/C \right)^\vee=A^B_{\cl^\dual}. \qedhere\]
\end{proof}

The following lemmas will be useful in proving the properties of this duality.

\begin{lemma}
\label{lem:duality}
Let $R$ be a complete local ring and $p$ be  a pair operation on a class of pairs of $R$-modules $\cP$. For any $(A,B) \in \cP$,
\[\left(\frac{B}{p(A,B)}\right)^\vee=p^\dual(\left(B/A\right)^\vee,B^\vee).\]
In particular, if $\cl$ is a closure operation then $((B/A)^\vee)_{\cl^\dual}^{B^\vee} = (B / A^\cl_B)^\vee$, and if $\int$ is an interior operation, then $((B/A)^\vee)^{\int^\dual}_{B^\vee} = (B / A_{\int}^B)^\vee$.
\end{lemma}

\begin{proof}
By definition, 
\[p^\dual((B/A)^\vee,B^\vee)=\left(\frac{B}{p(A,B)}\right)^\vee.\qedhere\] 
\end{proof}

\begin{lemma}[{\cite[Lemma 6.15]{ERGV-chdual}}]
\label{lemma:dualofintersectionissum}
Let $R$ be a complete Noetherian local ring.  
Let $B$ be an $R$-module such that it and all of its quotient modules are Matlis-dualizable.  Let $\{C_i\}_{i \in I}$ a collection of submodules of $B$. Then
\[ \left(\frac{B}{\sum_i C_i}\right)^\vee \cong \bigcap_i (B/C_i)^\vee\] and 
\[ \left(\frac{B}{\bigcap_i C_i}\right)^\vee \cong \sum_i (B/C_i)^\vee,\]
where all the dualized modules are considered as submodules of $B^\vee$.
\end{lemma}

\begin{lemma}\label{lem:dualitylift}
Let $R$ be a ring, $U \subseteq C \subseteq B$ $R$-modules,  $j: C \hookrightarrow B$ the inclusion map, $E$ an injective $R$-module, and ${(-)}^\vee$ the exact contravariant functor $\Hom_R(-,E)$.  Let $\pi = j^\vee: B^\vee \onto C^\vee$.  Then when $(C/U)^\vee$ (resp. $(B/U)^\vee$) is considered as a submodule of $C^\vee$ (resp. $B^\vee$) in the usual way, we have $\pi^{-1}((C/U)^\vee) = (B/U)^\vee$.  
\end{lemma}

\begin{proof}
For any $g \in B^\vee$ we have: \[
g \in \pi^{-1}((C/U)^\vee) \iff (g|_C)|_U = 0 \iff g|_U = 0 \iff g \in (B/U)^\vee. \qedhere
\]
\end{proof}

We give a number of results showing how the duality of Definition \ref{def:nonresidualdual} interacts with properties of pair operations.

\begin{prop}
\label{pr:dualityprops}
Let $R$ be a complete local ring, 
 and let $p$ be a pair operation on a class $\cP$ of pairs of Matlis-dualizable $R$-modules as in Definition \ref{def:pairops}. We have the following:
\begin{enumerate}
    \item\label{it:bidual} $p^{\dual \dual} = p$.
    \item\label{it:extint} If $p$ is extensive, then $p^\dual$ is intensive.
    \item\label{it:intext} If $p$ is intensive, then $p^\dual$ is extensive.
    \item\label{it:opsub} $p$ is order-preserving on submodules if and only $p^\dual$ is.
    \item\label{it:idem} $p$ is idempotent if and only if $p^\dual$ is.
    \item\label{it:clint} If $p$ is a closure operation, then $p^\dual$ is an interior operation.
    \item\label{it:intcl} If $p$ is an interior operation, then $p^\dual$ is a closure operation.
    \item\label{it:ressurfunc} $p$ is restrictable if and only if $p^\dual$ is surjection-functorial.
    \item\label{it:func} Assume $p$ is order-preserving on submodules.  Then $p$ is functorial if and only if $p^\dual$ is.
\end{enumerate}
\end{prop}

\begin{proof}
(\ref{it:bidual}): Using the definition and Lemma~\ref{lem:duality},  for any $(A,B) \in \cP$: \begin{align*}
p^{\dual \dual}(A,B) &= \left(B^\vee / p^\dual((B/A)^\vee, B^\vee)\right)^\vee = \left(B^\vee / (B / p(A,B))^\vee\right)^\vee \\
&= p(A,B)^{\vee \vee} = p(A,B),
\end{align*}
since everything is being considered as a submodule of $B$.

(\ref{it:extint}): Let $(A,B) \in \cP^\vee$.  Then $((B/A)^\vee, B^\vee) \in \cP$, so we have $(B/A)^\vee \subseteq p((B/A)^\vee, B^\vee) \subseteq B^\vee$.  Thus, we have \[
B^\vee \onto B^\vee / (B/A)^\vee \cong A^\vee \onto B^\vee / p((B/A)^\vee, B^\vee) \cong p^\dual(A,B)^\vee.
\]
Taking duals again, we have $p^\dual(A,B) = p^\dual(A,B)^{\vee \vee} \subseteq A^{\vee \vee} = A$.

(\ref{it:intext}): Let $(A,B) \in \cP^\vee$.  Then $((B/A)^\vee, B^\vee) \in \cP$, so $p((B/A)^\vee, B^\vee) \subseteq (B/A)^\vee$.  Thus, we have \[
A = A^{\vee \vee} = \left(\frac{B^\vee}{(B/A)^\vee}\right)^\vee \subseteq \left( \frac{B^\vee}{p((B/A)^\vee, B^\vee)} \right)^\vee = p^\dual(A,B).
\]

(\ref{it:opsub}): Note that by (\ref{it:bidual}), we need only prove one direction of the equivalence.  Accordingly, suppose $p$ is order-preserving on submodules.  Let $A\subseteq C \subseteq B$ be such that $(A,B), (C,B) \in \cP^\vee$.  Then we have $(B/C)^\vee \subseteq (B/A)^\vee \subseteq B^\vee$, and $((B/C)^\vee, B^\vee), ((B/A)^\vee, B^\vee) \in \cP$.  Then by assumption, we have $p((B/C)^\vee, B^\vee) \subseteq p((B/A)^\vee, B^\vee)$.  Thus we get a natural surjection \[
\frac{B^\vee}{p((B/C)^\vee, B^\vee)} \onto \frac{B^\vee}{p((B/A)^\vee, B^\vee)}.
\]
Applying Matlis duality and using the definition of $\dual$, it follows that: \[
p^\dual(C,B) = \left( \frac{B^\vee}{p((B/C)^\vee, B^\vee)}\right)^\vee \supseteq \left( \frac{B^\vee}{p((B/A)^\vee, B^\vee)}\right)^\vee = p^\dual(A,B).
\]

(\ref{it:idem}): Again by (\ref{it:bidual}), we need only prove one direction of the equivalence.  So suppose $p$ is idempotent, and $(A,B),(p^\dual(A,B),B) \in \cP^\vee$.  Then we have: \begin{align*}
    p^\dual(p^\dual(A,B),B)^\vee &= \frac{B^\vee}{p^\dual((B / p^\dual(A,B))^\vee,B^\vee)} & \\
    &= \frac{B^\vee}{p(p^{\dual \dual}((B/A)^\vee, B^\vee), B^\vee)} &\text{ by Lemma~\ref{lem:duality}} \\
    &= \frac{B^\vee}{p(p((B/A)^\vee, B^\vee), B^\vee)} &\text{ by (\ref{it:bidual})}\\
    &= \frac{B^\vee}{p((B/A)^\vee, B^\vee)} &\text{ since $p$ is idempotent} \\
    &= \left(\frac{B^\vee}{p((B/A)^\vee, B^\vee)}\right)^{\vee \vee} \\
    &=  p^\dual(A,B)^\vee.
\end{align*}
Another application of Matlis duality finishes the proof.

(\ref{it:clint}): This follows from (\ref{it:extint}), (\ref{it:opsub}), and (\ref{it:idem}).

(\ref{it:intcl}): This follows from (\ref{it:intext}), (\ref{it:opsub}), and (\ref{it:idem}).

(\ref{it:ressurfunc}): First suppose $p$ is restrictable. Let $A,B,D$ be 
such that $A,D$ are submodules of $B$ and $(A,B),((A+D)/D,B/D) \in \cP^\vee$.  Set $M := B^{\vee}$, $L := (B/A)^\vee \subseteq B^\vee$, and $N := (B/D)^\vee \subseteq B^\vee$.  Note that $L \cap N = (B / (A+D))^\vee$.  As a result, both $(L \cap N,N)$ and $(L,M)$ are in $\cP$. By restrictability of $p$, we have $p(L \cap N, N) \subseteq p(L,M) \cap N$.  Thus, we get a natural surjection $N / p(L\cap N, N) \onto N / (p(L,M) \cap N)$, and thus an inclusion \[
\left(\frac{N} {p(L,M) \cap N}\right)^\vee \subseteq \left(\frac{N} {p(L \cap N, N)}\right)^\vee.
\]
We then obtain the following sequence of equalities, with the lone inclusion on the third line below given by the inclusion shown above: \begin{align*}
\frac{p^\dual(A,B) + D}{D} &= \frac{\left(B^\vee / p((B/A)^\vee, B^\vee)\right)^\vee+D}{D} = \frac{(M / p(L,M))^\vee+D}{D} \\
&= \frac{(M/p(L,M))^\vee + (M/N)^\vee}{(M/N)^\vee} = \frac{ \left(\frac{M}{p(L,M) \cap N}\right)^\vee}{(M/N)^\vee} \\
&= \left(\frac{N}{p(L,M) \cap N}\right)^\vee \subseteq \left(\frac N {p(L \cap N, N)}\right)^\vee \\
&= p^\dual((N / (L \cap N))^\vee, N^\vee) = p^\dual((A+D)/D, B/D)
\end{align*}
where the fourth, sixth and seventh equalities follow from Lemma~\ref{lemma:dualofintersectionissum}.
Hence, $p^\dual$ is surjection-functorial.

Conversely, suppose $p^\dual$ is surjection-functorial.  Let $M \in \cM$ and suppose $L, N$ are submodules of $M$ such that $(L \cap N, N), (L,M) \in \cP$.  Set 
\[B := M^\vee,\ A := (M/L)^\vee \subseteq M^{\vee}=B,\ \text{and}\ D := (M/N)^\vee \subseteq M^\vee=B.\]  Note that $(N /(L \cap N))^\vee = (A+D)/D$ by Lemma~\ref{lemma:dualofintersectionissum}.  By surjection-functoriality of $p^\dual$, we have 
\[\displaystyle\frac{p^\dual(A,B) +D}D \subseteq p^\dual\left(\displaystyle\frac{A+D}D,\displaystyle\frac BD\right),\] 
giving us a natural surjection 
\[\displaystyle\frac{B/D}{(p^\dual(A,B)+D)/D} \onto \displaystyle\frac{B/D}{p^\dual\left(\frac{A+D}D, \frac BD\right)},\] and thus an inclusion \[
\left(\frac{B/D}{p^\dual\left(\frac{A+D}D, \frac BD\right)}\right)^\vee \subseteq \left( \frac{B/D}{(p^\dual(A,B)+D)/D}\right)^\vee
\]
We then obtain the following sequence of equalities, with the lone inclusion on the second line below given by the inclusion shown above: \begin{align*}
p(L \cap N, N) &= \left(\frac{N^\vee} {p^\dual\left(\left(\frac N {L \cap N}\right)^\vee, N^\vee \right)}\right)^\vee = \left(\frac{B/D}{p^\dual\left(\frac{A+D}D, \frac BD\right)}\right)^\vee \\
&\subseteq \left( \frac{B/D}{\left(\frac{p^\dual(A,B)+D}D\right)}\right)^\vee = \left(\frac B {p^\dual(A,B)+D}\right)^\vee\\
&= \left(\frac B {p^\dual(A,B)}\right)^\vee \cap (B/D)^\vee = p(L,M) \cap N
\end{align*}
where the second and fourth equalities follow from Lemma~\ref{lemma:dualofintersectionissum}.
Hence, $p$ is restrictable.

(\ref{it:func}): By (\ref{it:bidual}) and (\ref{it:opsub}), it suffices to prove that if $p$ is functorial, so is $p^\dual$.

Accordingly, suppose $p$ is functorial.  Then in particular it is surjection-functorial, whence by (\ref{it:ressurfunc}), $p^\dual$ is restrictable.  Then by Lemma~\ref{lem:restrict}(1), $p^\dual$ is order-preserving on ambient modules. On the other hand, since $p$ is functorial, it is order-preserving on ambient modules.  Since $p$ is also order-preserving on submodules, Lemma~\ref{lem:restrict}(2) then gives that $p$ is restrictable.  But then by (\ref{it:bidual}) and (\ref{it:ressurfunc}), $p^\dual$ is surjection-functorial.

Since $p^\dual$ is both order-preserving on ambient modules and surjection-functorial, it must be functorial.
\end{proof}

In \cite[Definition 3.1]{ERGV-chdual}, we defined a finitistic version of a closure operation. Here we use an alternative definition more suited to the nonresidual closure operations studied in this paper:

\begin{defn}\label{def:nonresidualfinitistic}
Let $\cl$ be a functorial closure operation (not necessarily residual) on a class $\cP$ of pairs of $R$-modules as in Definition \ref{def:pairops}. We define the \emph{finitistic version} $\cl_f$ of $\cl$ by
\begin{align*}
L^{\cl_f}_M = \bigcup \{(L \cap U)^{\cl}_U \mid U &\text{ is a \fg\ submodule of } M\\
&{\text{such that } (L \cap U, U) \in \cP}\}.
\end{align*}
\end{defn}

By the following lemma (a rephrased version of \cite[Lemma 3.2]{ERGV-chdual}), this definition agrees with the definition of a finitistic version given in \cite[Definition 3.1]{ERGV-chdual} when $\cl$ is residual.

\begin{lemma}[{\cite[Lemma 3.2]{ERGV-chdual}}]\label{lem:altfg}
Let $\cl$ be a functorial, residual closure operation on a class $\cP$ of pairs of $R$-modules as in Definition \ref{def:pairops}. Then for $(L,M) \in \cP$, \[
L^{\cl_f}_M:=\bigcup \{L_{N}^{\cl} \mid L \subseteq N \subseteq M,\ N/L \text{ finitely-generated and } (L,N) \in \cP\} .\]
\end{lemma}

Next we use our nonresidual duality to give several ways to view dual interiors on ideals.

\begin{prop}[cf. Theorem 3.3 of \cite{ERGV-chdual}]
Let $(R,\m,k)$ be a complete Noetherian local ring, let $\cl$ be a functorial closure operation (not necessarily residual) on the class of Artinian $R$-modules $\mathcal{A}$, and let $I$ be an ideal of 
$R$. We have
\begin{align*}
I^R_{\cl^\dual} 
 &\overset{(1)}{=}\ann_R\left((\ann_E I)^\cl_E\right)\overset{(2)}{=}\bigcap_{M \in \mathcal{A}} \ann_R\left((\ann_M I)_M^\cl\right) \\
 &\overset{(3)}{\subseteq} 
 I^R_{{\cl_f}^\dual} 
 \overset{(4)}{=}\ann_R\left((\ann_E I)^{\cl_f}_E\right)\overset{(5)}{=}\bigcap_{M \subseteq E \text{ f.g.}} \ann_R\left((\ann_M I)_M^\cl\right) \\
 &\overset{(6)}{=}\bigcap_{\lambda(M)<\infty} \ann_R\left((\ann_M I)_M^\cl\right) \overset{(7)}{\subseteq} \bigcap_{\lambda(R/J)<\infty} \ann_R\left((\ann_{R/J} I)_{R/J}^\cl\right)  \\
 &\overset{(8)}{\subseteq} \bigcap_{\lambda(R/J)<\infty} J:_R(J:_RI)_R^\cl.\\
\end{align*}
If $R$ is approximately Gorenstein and $\{J_t\}$ is a decreasing nested sequence of irreducible ideals cofinal with the powers of $\m$, then we get a further inclusion:
\[\bigcap_{\lambda(R/J)<\infty} J:_R(J:_RI)_R^\cl \overset{(9)}{\subseteq} \bigcap_t J_t:_R(J_t:_RI)_R^\cl.\]
\end{prop}

\begin{rem}
The key difference between this result and Theorem 3.3 of \cite{ERGV-chdual} is that (8) is an inclusion and not equality. To get equality, we need $\cl$ to be residual. Consequently, we cannot prove that (7) is an equality in the approximately Gorenstein case.

In theory, this is an obstruction to computing relative interiors as we computed absolute interiors in \cite{ERGV-chdual}, but as we will see in Section \ref{sec:examples}, we are still able to compute examples of basically empty interiors.
\end{rem}

\begin{proof}
\begin{itemize}
    \item[(1)] By definition,
    \[I^R_{\cl^\dual}=\left( \frac{R^\vee}{\left(\left(R/I\right)^\vee\right)_{R^\vee}^\cl }\right)^\vee=\left(\frac{E}{\left(\ann_E I)_E^\cl\right)} \right)^\vee=\ann_R\left((\ann_E I)_E^\cl\right).\]
    \item[(4)] Since $\cl_f$ is also a closure operation, we get the same equality for $\cl_f$.
    \item[(2)] The proof is identical to the proof of the second equality in \cite[Theorem 3.3]{ERGV-chdual}.
    \item[(3)] This holds as $\cl_f \le \cl$.
    \item[(5)] This proof is the same as the proof of the fourth equality in \cite[Theorem 3.3]{ERGV-chdual}, using our Definition \ref{def:nonresidualfinitistic} in place of Lemma 3.2 of \cite{ERGV-chdual}.
    \item[(6)] This proof is the same as the proof of the fifth equality in \cite[Theorem 3.3]{ERGV-chdual}.
    \item[(7)] This holds because every $R/J$ of finite length is an $R$-module of finite length.
    \item[(8)] Let $\pi:R \to R/J$ be the natural surjection. Then we have 
     \[\ann_{R/J} I=(J:_R I)/J=\pi(J:_R I).\]  Since $\cl$ is functorial, we have
    \begin{align*}
    \frac{(J:_R I)^\cl_R}{J}&=\pi((J:_R I)_R^\cl) \subseteq \left(\pi(J:_R I) \right)_{R/J}^\cl\\
    &=\left(\frac{J:_RI}{J} \right)_{R/J}^\cl=(\ann_{R/J} I)_{R/J}^\cl.
    \end{align*}
    Hence 
    \[\ann_R\left((\ann_{R/J} I)_{R/J}^\cl \right) \subseteq \ann_R \frac{(J:_RI)_R^\cl}{J}=J:_R(J:_RI)_R^\cl,
    \]
    giving us the desired inclusion.
    \item[(9)] This holds because $\lambda(R/J_t)<\infty$ for each $t$. \qedhere
\end{itemize}
\end{proof}

Next we define a relative version of a Nakayama interior (dual to Nakayama closure, recall Definition~\ref{def:Nakcl}), based on \cite{ERGV-chdual}

\begin{defn}
\label{nakayamainterior}
Let $(R, \m)$ be a local ring and $\int$ a relative interior operation on Artinian $R$-modules. We say that $\int$ is a Nakayama interior if for any Artinian $R$-modules $A \subseteq C \subseteq B$, if 
$(A:_C \m)^B_i\subseteq A$, then $A^B_{\int} = C^B_{\int}$ (or equivalently, $C^B_{\int} \subseteq A$).
\end{defn}

\begin{prop}
\label{pr:nakdual}
Let $(R, \m)$ be a complete local ring. Let $\cl$ be a closure operation, not necessarily residual, on the class of finitely generated $R$-modules. Then $\cl$ is a Nakayama closure if and only if $\cl^\dual$ is a Nakayama interior.
\end{prop}

\begin{proof}
Suppose $\cl$ is a Nakayama closure.  Let $A \subseteq C \subseteq B$ be Artinian $R$-modules, and suppose $(A :_C \m)_{\cl^\dual}^B \subseteq A$.  Set $M := B^\vee$, $N := (B/A)^\vee$, and $L := (B/C)^\vee$.  We consider $L \subseteq N \subseteq M$ in the usual way.  By \cite[Lemma 5.4]{ERGV-chdual}, we have $(A :_C \m)^\vee = \frac{M}{L+\m N}$.  Thus, we have \[
(M/N)^\vee = A \supseteq (A :_C \m)_{\cl^\dual}^B = \left(\frac{B^\vee}{((B / (A :_C\m))^\vee)^\cl_{B^\vee}} \right)^\vee = \left(\frac{M}{(L + \m N)^\cl_M}\right)^\vee.
\]
Applying Matlis duality, we obtain a natural surjection 
\[M/N \onto M / (L + \m N)^\cl_M, \text{ whence } N \subseteq (L + \m N)^\cl_M.\]  The fact that $\cl$ is a Nakayama closure then implies that $N \subseteq L^\cl_M$, whence $M/N \onto M/ L^\cl_M$, so that $(M/L^\cl_M)^\vee \subseteq (M/N)^\vee$.  Therefore, \[
C_{\cl^\dual}^B = \left(\frac{B^\vee}{((B/C)^\vee)^\cl_{B^\vee}}\right)^\vee = \left(\frac{M}{L^\cl_M}\right)^\vee  \subseteq (M/N)^\vee = A.
\]

Conversely, suppose $\int= \cl^\dual$ is a Nakayama interior.  Let $L \subseteq N \subseteq M$ be finitely generated $R$-modules, and suppose $N \subseteq (L + \m N)^\cl_M$.  Let $A = (M/N)^\vee$ and $C := (M/L)^\vee$, thought of as submodules of $B := M^\vee$ in the usual way.  Then $M/N \onto M / (L + \m N)^\cl_M$, so that $(M / (L + \m N)^\cl_M)^\vee \subseteq (M/N)^\vee$.  Hence by \cite[Lemma 5.4]{ERGV-chdual} again, we have \[
(A :_C \m)_{\int}^B = \left(\frac{M}{(L + \m N)^\cl_M}\right)^\vee \subseteq (M/N)^\vee = A.
\]
Then since $\int$ is a Nakayama interior, it follows that $C_{\int}^B \subseteq A$.  Thus,\[
N= (B/A)^\vee \subseteq (B/C_{\int}^B)^\vee = ((B/C)^\vee)^\cl_{B^\vee} = L^\cl_M,
\]
where the second to last equality comes from Lemma \ref{lem:duality} and Proposition \ref{pr:dualityprops}.
\end{proof}

We end this section by showing that the identity operation is an example of a relative Nakayama interior, showing that \cite[Lemma 5.3]{ERGV-chdual} works in the relative case. 

\begin{prop}[{cf. \cite[Lemma 5.3]{ERGV-chdual}}]
Let $R$ be a Noetherian local ring. The identity operation is a relative Nakayama interior on the class of Artinian $R$-modules.
\end{prop}

\begin{proof}
The proof is identical to the proof of \cite[Lemma 5.3]{ERGV-chdual} with the instances of $B$ in the proof replaced with $C$ such that $A \subseteq C \subseteq B$.
\end{proof}

\section{Basically full closure and its dual interior: basically empty interior}\label{sec:bfcbei}

In this section we investigate basically full closure, one of the main examples of this paper. Basically full closure is a nonresidual Nakayama closure, so we are able to use our methods to find its dual interior operation.

\begin{defn}
\cite[Definition 2.1]{HRR-bf}
\label{def:basfull} Let $(R,\m)$ be a quasilocal ring. If $L \subseteq M$ are $R$-modules with $L$ finitely generated, we say that $L$ is \emph{basically full in $M$} if for any finitely generated $R$-module $N$ with $L \subsetneqq N \subseteq M$, no miminal generating set for $L$ can be extended to one for $N$.
\end{defn}

\begin{defn}\label{def:Jcol}
Let $R$ be a commutative ring and let $J$ be an ideal of $R$.  Then for any submodule inclusion $L \subseteq M$, the \textit{$J$-basically full closure} of $L$ in $M$ is given by \[
\Jcol JLM := (JL :_M J).
\]
\end{defn}

\begin{rem}
\label{rem:twobasfull}
When $R$ is quasilocal and $M/L$ has finite length, it is shown in \cite[{Theorem 4.2}]{HRR-bf} that $(\m L :_M \m)$ is the \emph{unique} smallest basically full module that contains $L$.  Hence, in this case these authors speak of the \emph{basically full closure} of $L$ in $M$. In our notation, this is the $\m$-basically full closure of $L$ in $M$.
\end{rem}

Notice that the operation $(L,M) \mapsto (J L :_M J)$ is a closure operation, whether or not we use the ideal $J=\m$ and regardless of whether $M/L$ has finite length (see Proposition \ref{jcolnakayama} for a proof). 
This closure operation has been discussed previously. Rush discusses $J$-basically full closure of ideals in \cite{Rush-contracted}, where he calls an ideal $I$ $J$-basically full if $(JI:J)=I$. Vraciu and the third-named author reference this closure in the ideal setting in \cite{VaVr}, where they call this closure $\Jcolsym J$.  To follow the notation of Heinzer, Ratliff and Rush, we make the following definition:

\begin{defn}
We say that a submodule $L$ of $M$ is \textit{$J$-basically full} if $L=L_M^{\Jcolsym J}$.  
\end{defn}

The next proposition is essentially due to a test for integral closure from Vasconcelos \cite[Proposition 1.58]{Vascon05} and shows the importance of $J$-basically full closure.  However, a more general result is due to Ratliff and Rush \cite[Proposition 3.1]{RRdeltamod}. 

\begin{prop}\label{pr:vascriterion}
Let $R$ be a domain and $I$ an ideal in $R$.  $I$ is integrally closed if and only if $\Jcol{J} I R=I$ for all ideals $J$ in $R$. 
\end{prop}
 We include an example here of an $\m$-basically full ideal which is not $\m^2$-basically full in $R$, but is $\m^2$-basically full in itself.  This example illustrates the importance of the choice of $J$ and of the ambient module in defining $J$-basically full closure. 

\begin{example}\label{ex:bfex}
Let $R=k[[x,y]]$ with $\m =(x,y)$.  The ideal $I=(x^3,x^2y^2,y^3)$ is $\m$-basically full (i.e. basically full) by \cite[Example 9.1]{HRR-bf}.  However, $\m^2 I = \m^5$, so that $I\subsetneq \m^3 = \Jcol {{\m^2}}{I}{R}$. On the other hand, viewing $I$ as a submodule of itself, we of course have $\Jcol {{\m^2}}{I}{I}=I$.
\end{example}

\begin{prop}
\label{jcolnakayama}
For any ideal $J$, the operation $\Jcolsym J$ is a functorial closure operation.  Moreover, when $(R,\m)$ is local and $J$ is finitely generated, then $\Jcolsym J$ is Nakayama when considered over the class of finitely generated modules and their submodules.
\end{prop}

\begin{proof}
Let $L \subseteq M$ be $R$-modules.  Let $z\in L$.  Then $Jz \subseteq JL$, so $z\in \Jcol JLM$ by definition.  Hence, $\Jcolsym J$ is extensive.

Let $L \subseteq N \subseteq M$ be $R$-modules.  Let $z \in \Jcol JLM$.  Then $Jz \subseteq JL \subseteq JN$ since $L \subseteq N$, whence $z \in (JN :_M J) = \Jcol JNM$.  Hence $\Jcolsym J$ is order-preserving.

Let $L \subseteq M$ be $R$-modules and let $z\in \Jcol J {\left(\Jcol JLM\right)} M$.  Then $Jz \subseteq J\left(\Jcol JLM\right) = J(JL :_M J) \subseteq JL$.  Hence $z \in (JL :_M J) = \Jcol JLM$.

To see functoriality, let $L \subseteq M$ and $N$ be $R$-modules and let $\phi: M \ra N$ be an $R$-linear map.  Let $z\in \Jcol JLM = (JL :_M J)$. Then $Jz \subseteq JL$, so by $R$-linearity we have $J\phi(z) \subseteq J \phi(L)$, whence $\phi(z) \in (J \phi(L) :_N J) = \Jcol J{\phi(L)}N$.

To prove the Nakayama property, let $L \subseteq N \subseteq M$ be finitely generated modules, and $N \subseteq \Jcol J{(L + \m N)}M$.  Then we have $JN \subseteq J(L + \m N) = JL + \m JN$.  It follows that \[JN / JL \subseteq (JL + \m JN) / JL = \m (JN / JL).\]  
Since $JN/JL$ is finitely generated, by Nakayama's lemma we have $JN/JL = 0$, whence $N \subseteq (JL :_M J) = \Jcol JLM.$
\end{proof}

\begin{rem}
Note that $\Jcolsym J$ is in general \emph{not} a residual closure operation.  If $\pi: M \onto M/L$, then $\pi^{-1}(\Jcol J 0 {M/L}) = (L:_M J)$, which is in general not equal to $\Jcol J L M = (JL :_M J)$. For example, if we take $M=R$ to be Noetherian local, with nonzero maximal ideal, and $J=L=\m$, then $L:_M J=\m :_R \m =R$, but $JL:_M J=\m^2:_R \m \ne R$.
\end{rem}

\subsection{Dual of $\Jcolsym J$}
\label{subsec:jbfdual}

 In this subsection we use the results of Section \ref{sec:nonresidualdual} to define a dual for $\Jcolsym J$ called the $J$-basically empty interior. The reason for the name will be explained further in Section \ref{sec:cogen}. 

\begin{defn}
\label{def:jbei}
Let $J$ be an ideal and $L \subseteq M$ be $R$-modules.  We define the \emph{$J$-\bemp \  interior of $L$ with respect to $M$} as $\Jintrel JML := J (L :_M J)$.

 When we refer to a \emph{$J$-basically empty submodule} $L$ of $M$, we mean that $L$ is open in $M$ with respect to the $J$-basically empty interior, i.e. $\Jintrel JML=L$.
 \end{defn}

 \begin{prop}
 \label{pr:jbenakayama}
 Let $R$ be a commutative ring and $J$ an ideal of $R$. Then the $J$-basically empty interior is a functorial interior operation on $R$-modules. If $(R,\m)$ is Noetherian local, then $\Jintrelsym J$ is a Nakayama interior.
 \end{prop}

 \begin{proof}
 First we prove that $\Jintrelsym J$ is an interior operation. Let $L \subseteq M$ be $R$-modules. Since $J(L:_M J) \subseteq L$, $\Jintrelsym J$ is intensive.
 We have
 \[(L^M_{\Jintrelsym J})^M_{\Jintrelsym J}=J(L^M_{\Jintrelsym J}:_M J) \supseteq J(L:_M J)=L^M_{\Jintrelsym J}.\]
 Hence $\Jintrelsym J$ is idempotent.
 Now let $L \subseteq N \subseteq M$. Since $L:_M J \subseteq N:_M J$, $L^M_{\Jintrelsym J} \subseteq N^M_{\Jintrelsym J}$. So $\Jintrelsym J$ is a closure operation.
 
 Next we prove that $\Jintrelsym J$ is functorial. Suppose $L \subseteq M$ and $\phi:M \rightarrow N$ is an $R$-module homomorphism.  First note that if $x \in (L:_MJ)$, then $Jx \subseteq L$.  Applying $\phi$, we get $J\phi(x)=\phi(Jx) \subseteq \phi(L)$ or $\phi(x) \in (\phi(L):_N J)$.  Any element of $\Jintrel J ML$ has the form $\sum\limits_{i=1}^n j_ix_i$ where $j_i \in J$ and $x_i \in (L:_M J)$ for $1 \leq i \leq n$ and 
\[\phi(\sum\limits_{i=1}^n j_ix_i)=\sum\limits_{i=1}^n j_i \phi(x_i) \in J(\phi(L):_N J)=\Jintrel J N{\phi(L)}.\]
Hence, $\phi(\Jintrel JML) \subseteq \Jintrel J N{\phi(L)}$ completing the proof that $\Jintrelsym J$ is functorial.

Finally we prove that $\Jintrelsym J$ is a Nakayama interior when $(R,\m)$ is Noetherian local.  Let $L \subseteq N \subseteq M$ be Artinian $R$-modules such that $\Jintrel J M {(L:_N \m)} \subseteq L$.  Expanding, we have $L \supseteq J((L :_N \m) :_M J) = J \cdot ((L :_M \m J) \cap (N :_M J))$.  Thus, \[
(L :_M \m J) \cap (N :_M J) \subseteq (L:_M J).
\]
Hence we have \begin{align*}
\Soc \left( \frac{N:_M J}{L :_M J}\right) &= \frac{(L:_M J) :_{(N :_M J)} \m}{L:_M J} = \frac{((L :_M J) :_M \m) \cap (N :_M J)}{L :_M J} \\
&= \frac{(L :_M \m J) \cap (N:_M J)}{L:_M J} = 0
\end{align*}
But a nonzero finitely cogenerated (i.e., Artinian) $R$-module has a nonzero socle (see Proposition~\ref{pr:fcoglam}).  Hence, $\frac{N:_M J} {L:_M J}=0$, so that $(L :_M J) = (N:_M J)$, whence: \[\Jintrel J M L = J(L :_M J) = J(N :_M J) = \Jintrel J M N.
\]  Therefore, $\Jintrelsym J$ is Nakayama.
 \end{proof}

 In order to prove that $\Jintrelsym J={\Jcolsym J}^\dual$, we need the following technical results:

\begin{thm}\label{thm:coldual}
Let $R$ be a complete Noetherian local ring.  Let $L \subseteq M$ be $R$-modules that are both \fg\ or both Artinian, and let $I$ and $J$ be ideals. 
Let $B := M^\vee$ and $A := (M/L)^\vee$, considered as a submodule of $M^\vee$ in the usual way.  Then the following hold: \begin{enumerate}
\item $\left(\frac{M}{JL :_M I}\right)^\vee = I(A:_BJ)$, and
\item $\left(\frac{B}{I(A:_BJ)}\right)^\vee = (JL :_M I)$.
\end{enumerate}
\end{thm}

\begin{proof}
For (a), recall that $(M/JL)^\vee = (A:_B J)$ by \cite[Lemma 5.4]{ERGV-chdual}. Say $I = (x_1, \ldots, x_n)$.  Then  we have the following exact sequence \[
\xymatrix{
0 \ar[r]& \frac{JL:_MI}{JL} \ar[r]^j&\frac{M}{JL} \ar[rr]^{\phi=\left(\begin{matrix}
x_1\\ \vdots \\ x_n
\end{matrix}\right)} \ar@{->>}[dr] & & \bigoplus_{i=1}^n \frac{M}{JL}\\
& &  & \frac{M}{JL:_M I} \ar@{^(->}[ur]
}
\]
That is, $M/(JL:_MI)$ is the cokernel of $j$, and hence is isomorphic to the image of $\phi$.

Applying Matlis duality, we have \[
\xymatrix{
0 & \left( \frac{JL:_MI}{JL}\right)^\vee \ar[l] & (A :_B J) \ar[l] && \bigoplus_{i=1}^n (A:_B J) \ar[ll]_{\phi^\vee = (x_1\ \cdots\ x_n)} \ar[dl]\\
& & & \left(\frac{M}{JL:_MI}\right)^\vee \ar[ul]&
}
\]
Hence, $\left(\frac{M}{JL:_MI}\right)^\vee = \im (\phi^\vee) = I(A:_BJ)$.

For (b), then, we have \[
\left(\frac{B}{I(A:_BJ)}\right)^\vee = \left(\frac{M^\vee}{(M / (JL :_MI))^\vee}\right)^\vee = ((JL:_MI)^\vee)^\vee = JL:_MI,
\]
since $JL :_MI$ is a Matlis-dualizable module.
\end{proof}

Now we can prove that $\Jintrelsym J$ is the interior operation dual to $\Jcolsym J$.

\begin{thm}\label{thm:Jcoldual}
Let $R$ be a complete local ring. Then ${\Jcolsym J}^\dual(A,B)=\Jintrel JBA$ for any pair of Artinian or \fg\ $R$-modules $A \subseteq B$.
\end{thm}

\begin{proof}
Let $A \subseteq B$ be $R$-modules and set $M=B^\vee$ and $L=\left(B/A\right)^{\vee} \subseteq M$. By definition,
\[{\Jcolsym J}^\dual(A,B)=\left(\frac{B^\vee}{
\Jcol J{\left(\left(\frac{B}{A}\right)^\vee\right)}{B^\vee}
} \right)^\vee.
\]
Substituting, we get
\[\left(\frac{M}{
\Jcol JL{M}
} \right)^\vee
=\left(
\frac{M}{JL:_M J}
\right)^\vee=J(A:_B J)=\Jintrel JBA,
\]
by Theorem \ref{thm:coldual} with $I=J$.
\end{proof}

We include an example of some $J$-basically empty interiors of ideals in a two dimensional regular local ring.  

\begin{example}\label{ex:beintid}
Let $R=k[[x,y]]$ with $\m =(x,y)$ and $I=(x^3,y^3)$. The $\m$-basically empty interior of $I$ in $R$ is $\Jintrel \m RI=(x^4,x^3y,xy^3,y^4)$.  However, $\Jintrel {{\m^2}}{R}{I}=\m^4$. 
\end{example}

We now give an example of the $\m$-basically full closure of a submodule of the injective hull of the residue field again noting the duality with the basically empty interior of an ideal in the ring.

\begin{example}
\label{ex:basfullinE}
Let $R=k[[x,y]]$ with $\m =(x,y)$.  
We know 
$E=E_R(k)$ is isomorphic to the $R$-module of polynomials in $k[x^{-1},y^{-1}]$  
where the action of a monomial $x^ny^m$ is given by 
\[
x^ny^m \cdot x^{-r}y^{-s}=\begin{cases} 0 &\text{ if } n > r \text{ or } m > s\\
x^{-(r-n)}y^{-(s-m)} &\text{ if } n\leq  r \text{ and } m \leq s.\\
\end{cases}\]
Consider 
\[N=kx^{-1}y^{-1}+kx^{-1}y^{-2}+kx^{-1}y^{-3}+kx^{-2}y^{-1}+kx^{-3}y^{-1}=(0:_E (x^3,xy,y^3)).\]
First note that $\m N=kx^{-1}y^{-1}+kx^{-1}y^{-2}+kx^{-2}y^{-1}=(0:_E \m^2)$.  Then 
\[\Jcol \m N E=(\m N:_E\m)=N+kx^{-2}y^{-2}.\]
We have \[\left(\displaystyle\frac{E}{\Jcol \m N E}\right)^{\vee}=\m^3=\m((x^3,xy,y^3):_R\m)=\Jintrel \m R {(x^3,xy, y^3)}. \]
Note that $(x^3,xy,y^3)=\Jcol {\m}{(x^3,xy,y^3)}R$ is $\m$-basically full but not $\m$-basically empty.
\end{example}

\section{Cogenerators and basically empty modules}
\label{sec:cogen}

We give a dual to the notion of basic fullness from \cite{HRR-bf}; the main result is Theorem \ref{thm:be}. In order to do this, we use the concept of finite cogeneration. We end the section by giving a criterion for basic emptiness of a submodule $A \subseteq B$ in terms of submodules $C$ which are covered by $A$.

From now until after Definition \ref{def:basicemptiness}, 
 $R$ will be an \emph{associative} ring (and not necessarily commutative, Noetherian, or local).  All modules will be left $R$-modules.

\begin{defn}\label{def:vamcogen}
(See \cite{Vam} or \cite[19.1-19.2 and Exercise 19.7]{Lam99}) A module $M$ is called \emph{\fcog} if there is a finite list (possibly with repetition) of simple $R$-modules $S_1, \ldots, S_t$ such that $E(M) \cong \bigoplus_{i=1}^t E(S_i)$.

We use $\Soc M$ to denote the \emph{socle} of a module $M$, defined to be the sum of all the simple submodules of $M$.
\end{defn}

Here are some equivalent definitions and useful properties of finite cogeneration:

\begin{prop}\label{pr:fcoglam}  For any associative ring $R$ and $R$-module $M$.
\begin{enumerate}
    \item \cite[19.1(3)]{Lam99} $M$ is \fcog\ if and only if $\Soc M$ is \fg\ and $M$ is an essential extension of $\Soc M$. 
    \item \cite[Exercise 19.8]{Lam99} If $R$ is a commutative Noetherian ring, then $M$ is \fcog\ if and only if it is Artinian. 
    \item \cite[Exercise 19.2]{Lam99} If $M$ is \fcog, then so are submodules of $M$ and essential extensions of $M$. 
    \item \cite[Exercise 19.5]{Lam99} A finite direct sum of modules $\sum_{i=1}^n M_i$ is \fcog\ if and only if $M_i$ is \fcog\ for all $1 \le i \le n$. 
\end{enumerate}
\end{prop}

Next we frame the notion of finite cogeneration of a module in terms of $R$-linear maps to injective hulls of simple $R$-modules. This works in a level of generality similar to the discussion of finite cogeneration in \cite{ERGV-chdual}.

\begin{defn}\label{def:cogset}
Let $M$ be an $R$-module.  A (finite) \emph{cogenerating set} for $M$ is a list of $R$-linear maps $g_i: M \rightarrow E(S_i)$, $1\leq i \leq t$ for some nonnegative integer $t$, such that each $S_i$ is a simple $R$-module, and such that $\bigcap_{i=1}^t \ker g_i=0$.

We say that a cogenerating set for $M$ is \emph{minimal} if for all $1\leq j \leq t$, $g_1, \ldots, \widehat{g_j}, \ldots, g_t$ is \emph{not} a cogenerating set for $M$.  That is, for each $1\leq j \leq t$, we have $\bigcap_{i\neq j} \ker g_i \neq 0$.
\end{defn}

\begin{lemma}\label{lem:fcogequiv}
Let $M$ be an $R$-module.  Let $S_1, \ldots, S_t$ be a finite list of (not necessarily distinct) simple $R$-modules. The following are equivalent: \begin{enumerate}
    \item There is an injective $R$-linear map $\alpha: M \into \bigoplus_{i=1}^t E(S_i)$.
    \item There are $R$-linear maps $g_i: M \rightarrow E(S_i)$, $1\leq i \leq t$ that cogenerate $M$.
\end{enumerate}
\end{lemma}

\begin{proof}
First assume (2). Let $\alpha:M \to \bigoplus_{i=1}^t E(S_i)$ be given by the column vector $[g_1 g_2 \ldots g_t]^\intercal$. We have
\[
\bigcap_{i=1}^t \ker g_i = \ker \left[ \begin{matrix} g_1 \\ \vdots \\ g_t\end{matrix}\right],
\]
which gives the result.

Now assume (1). Set $g_i$ to be the $i$th entry of $\alpha$, and the result follows.
\end{proof}

\begin{prop}\label{pr:mincogequiv}
Let $M$ be an $R$-module.   Let $S_1, \ldots, S_t$ be a finite list of (not necessarily distinct) simple $R$-modules. The following are equivalent: \begin{enumerate}
    \item\label{it:injess} There is an injective $R$-linear map $M \into \bigoplus_{i=1}^t E(S_i)$, such that the target is an essential extension of the image of $M$.
    \item\label{it:congsimple} $E(M) \cong \bigoplus_{i=1}^t E(S_i)$.
    \item\label{it:mcog} There are $R$-linear maps $g_i: M \rightarrow E(S_i)$, $1\leq i \leq t$ that are a minimal cogenerating set for $M$.
\end{enumerate}
\end{prop}

\begin{proof}
\ (\ref{it:injess} $\iff$ \ref{it:congsimple}): This is standard in the theory of injective modules; see for example \cite[Corollary 3.33]{Lam99}.

\noindent (\ref{it:congsimple} $\implies$ \ref{it:mcog}): For each $j$, let $g_j$ denote the composition 
\[g_j:M \into E(M) \overset{\cong}{\ra} \bigoplus_{i=1}^t E(S_i) \overset{\pi_j}\onto E(S_j),\] where the leftmost map is the inclusion and the rightmost map is the projection $\pi_j$.  Let $\phi: E(M) \overset{\cong}{\ra} \bigoplus_{i=1}^t E(S_i)$ be the given isomorphism.  Since $\bigoplus_{i=1}^t E(S_i)$ is isomorphic to an injective hull for $M$, $\bigcap_{j=1}^t \ker g_j=0$. By symmetry, it is enough to show that $\bigcap_{j=2}^t \ker g_j \neq 0$.  Accordingly, choose a nonzero element $z\in E(S_1)$.  Let $t := \phi^{-1}\left(\left[\begin{matrix} z \\ 0 \\ \vdots \\ 0 \end{matrix}\right]\right) \in E(M)$.  By essentiality, there is some $r\in R$ such that $0 \neq rt \in M$.  Then for any $j\geq 2$, $g_j(rt) = \pi_j\left(\left[\begin{matrix} rz \\ 0 \\ \vdots \\ 0 \end{matrix}\right]\right) = 0$.

\noindent (\ref{it:mcog} $\implies$ \ref{it:congsimple}): Let $\psi := \left[\begin{matrix} g_1 \\ \vdots \\ g_t \end{matrix}\right]: M \ra \bigoplus_{i=1}^t E(S_i)$.  
Let $\widetilde \psi: E(M) \ra \bigoplus_{i=1}^t E(S_i)$ be a lifting, given by the fact that $\bigoplus_{i=1}^t E(S_i)$ is injective.  Moreover, since $E(M)$ is injective as well, we have $\bigoplus_{i=1}^t E(S_i) \cong E(M) \oplus \coker \widetilde \psi$.  Hence all modules in sight have \fg\ (hence finite length) socles.  It is enough to show that the length of $\Soc M$ is at least $t$, since $\coker \widetilde \psi$, if nonzero, will have a nonzero socle (since it is \fcog, and hence an essential extension of its socle).

Accordingly, for each $1\leq j \leq t$, let $0 \neq y_j \in \bigcap_{i\neq j} \ker g_i$.  Then $\psi(y_j) = \left[\begin{matrix} 0 \\ \vdots \\ z_j \\ \vdots \\ 0\end{matrix}\right]$, where $0\neq g_j(y_j) = z_j \in E(S_j)$, and with all other entries of the vector being zero.  By taking multiples, we may assume that in fact $z_j \in S_j \subseteq E(S_j)$. Let $N_j := R y_j$.  Then we have $g_j(N_j)$ is a nonzero submodule of $S_j$, so that by simplicity, we have $g_j(N_j) = S_j$.  Moreover, for any $u\in N_j$ with $g_j(u) = 0$, we have $u \in \bigcap_{k=1}^t \ker g_k = 0$; hence $g_j |_{N_j}$ is injective.  It follows that $N_j \cong S_j$, with the isomorphism being given by the restriction of $g_j$.

Moreover, for each $j$, we have $N_j \cap \sum_{i\neq j} N_i = 0$. To see this, simply note that $N_j \subseteq \bigcap_{i\neq j} \ker g_i$, and that each $N_i \subseteq \ker g_j$ whenever $i\neq j$, so that $\sum_{i\neq j} N_i \subseteq \ker g_j$.  Hence $N_j \cap \sum_{i\neq j} N_i \subseteq \bigcap_{k=1}^t \ker g_k = 0$.  It follows that $\sum_{i=1}^t N_i$ is a direct sum, with each summand isomorphic to the corresponding $S_i$.  Thus $\len(\Soc M) \geq \len(\sum_{i=1}^t N_i) = \sum_{i=1}^t \len(S_i) = t$.
\end{proof}

We now give a dual to basic fullness (see Definition \ref{def:basfull}):

\begin{defn}
\label{def:basicemptiness}
Let $R$ be a ring and let $A \subseteq B$ be left $R$-modules such that $B$ (and hence also $A$) is finitely cogenerated.  We say that $A$ is \emph{\bemp} 
\emph{in} $B$ if for any proper submodule $C \subsetneq A$, no minimal cogenerating set of $B/A$ extends to a minimal cogenerating set of $B/C$.
\end{defn}

We return to the situation where $(R,\m)$ is a commutative Noetherian local ring.

\begin{defn}[{\cite[Definition 6.6]{ERGV-chdual}}]\label{def:cog}
Let $R$ be a Noetherian local ring, $L$ an $R$-module, and $g_1,\ldots,g_t \in L^\vee$. We say that the \textit{quotient of $L$ cogenerated by $g_1,\ldots,g_t$} is $L/\left(\bigcap_i \ker(g_i)\right)$.

We say that $L$ is \textit{cogenerated by $g_1,\ldots,g_t$} if $\bigcap_i \ker(g_i)=0$.

We say that a cogenerating set for $L$ is \textit{minimal} if it is irredundant, i.e., for all $1 \le j \le t$, $\bigcap_{i \ne j} \ker(g_i) \ne 0$.
\end{defn}

\begin{rem}
The question arises: Is the V\'amos notion of \fcog\ given in Definition~\ref{def:cogset} compatible with the notion of cogenerating sets given in Definition~\ref{def:cog}?  
Fortunately, it is:
in this case, the only simple $R$-module is $k$, so Definition~\ref{def:cogset}  specializes to Definition \ref{def:cog}. 
\end{rem}

\begin{rem}
As we did in Section \ref{sec:bfcbei} for basically full, when we say a submodule $L \subseteq M$ is basically empty, we are referring to the above condition involving cogenerators. If we say a submodule $L$ of $M$ is $J$-basically empty, we mean $L=\Jintrel JML$ (see Definition \ref{def:jbei}).
\end{rem}

We recall here some additional results of \cite{HRR-bf} and dualize them in the next theorem.

\begin{thm}\label{thm:bfhrr}
\ \cite{HRR-bf} Let $(R,\m)$ be a Noetherian local ring and let $L\subseteq M$ be finitely generated $R$-modules. \begin{enumerate}
    \item \ [Theorem 2.3] The following are equivalent: \begin{enumerate}
        \item Some minimal generating set of $L$ extends to one of $M$.
        \item Every minimal generating set of $L$ extends to one of $M$.
        \item $\m L = L \cap \m M$
        \end{enumerate}
    \item\ [Corollary 2.5] $L$ is \bfull\ in $M \iff L/\m L$ is \bfull\ in $M/\m L$.
    \item\ [Theorem 2.6] If $L$ is \bfull\ in $M$, then $\len(M/L)<\infty$.
    \item\ [Theorem 2.12] If $\len(M/L) < \infty$, then $L$ is \bfull\ in $M \iff L = (\m L :_M \m)$  $($i.e. $L = \Jcol \m L M)$.
    \item\ [Corollary 2.15] If $\len(M/L)<\infty$, then for each positive integer $n$, $(\m^nL:_M\m^n)$ (i.e. $\Jcol {{\m^n}} L M$) is \bfull\ in $M$.
\end{enumerate}
\end{thm}

We have the following dual theorem.

\begin{thm}\label{thm:be}
Let $(R,\m)$ be a Noetherian local ring and let $A \subseteq B$ be Artinian $R$-modules. \begin{enumerate}
    \item\label{it:cogensart} The following are equivalent: \begin{enumerate}
        \item\label{it:cogensart1} Some minimal cogenerating set of $B/A$ extends to one of $B$.
        \item\label{it:cogensart2} Every minimal cogenerating set of $B/A$ extends to one of $B$.
        \item\label{it:cogensart3} $(A :_B \m) = (0 :_B \m) + A$.
    \end{enumerate}
    \item $A$ is \bemp\ in $B \iff A$ is \bemp\ in $(A :_B \m)$.
    \item If $A$ is \bemp\ in $B$, then $\len(A)<\infty$.
    \item If $\len(A)<\infty$, then $A$ is \bemp\ in $B \iff A = \m (A:_B \m)$  $($i.e. $A = \Jintrel \m B A)$.
    \item If $\len(A) < \infty$, then for each positive integer $n$, $\m^n(A:_B\m^n)$ (i.e. $\Jintrel {{\m^n}} B A$) is \bemp\ in $B$.
\end{enumerate}
\end{thm}

Before proving this result, we 
need another lemma.  In order to use Matlis duality we pass to the completion.

\begin{lemma}\label{lem:bfbe}
Let $A \subseteq B$ be Artinian $R$-modules, where $(R,\m)$ is a Noetherian local ring.  Then $A$ is \bemp\ in $B \iff (B/A)^\vee$ is \bfull\ in $B^\vee$ as $\hat{R}$-modules.
\end{lemma}

\begin{proof}
Let $L := (B/A)^\vee \subseteq B^\vee =: M$.

Suppose $A$ is \bemp\ in $B$.  Let $N$ be an $\hat R$-submodule of $M$ that properly contains $L$.  Let $C := (M/N)^\vee$.  Then $C$ is a proper submodule of $A$.  Let $g_1, \ldots, g_t$ be a minimal generating set of $L$.  Suppose it extends to a minimal generating set $g_1, \ldots, g_t, g_{t+1}, \ldots, g_s$ of $N$.  Then these same elements form a minimal cogenerating set of $B/A$ that extends to one of $B/C$, contradicting the basic emptiness of $A$ in $B$.  Thus, $L$ is basically full in $M$.

Conversely, suppose $L$ is \bfull\ in $M$ as $\hat R$-modules.  Let $C$ be a proper $R$- (hence $\hat R$-)submodule of $A$.  Let $N := (B/C)^\vee$.  Then $N$ is an $\hat R$-submodule of $M$ that properly contains $L$.  Suppose $g_1, \ldots, g_t: B \ra E$ are a minimal cogenerating set of $B/A$.  Then they are a minimal generating set for $(B/A)^\vee = L$, so by basic fullness, they cannot extend to a minimal generating set for $(B/C)^\vee = N$, i.e., a minimal cogenerating set for $B/C$.  Thus, $A$ is \bemp\ in $B$.
\end{proof}

\begin{proof}[Proof of Theorem~\ref{thm:be}]
First recall that $(A :_B\m) = (A :_B \hat{\m})$, which follows from the fact that $B/A$ is an $\hat R$-module and Hom-tensor adjointness.  Also, for any Artinian module $C$, we have $\m C = \hat{\m}C$ by similar considerations.  Hence, we may assume for this proof that $(R,\m)$ is \emph{complete}.

For the rest of the proof, let $M := B^\vee$, and let $L := (B/A)^\vee$, considered as a submodule of $M$.

To prove (\ref{it:cogensart}), recall by \cite[Lemma 6.8]{ERGV-chdual} that (minimal) cogenerating sets of $B$ (resp. $B/A$) are the same as (minimal) generating sets of $M$ (resp. $L$).  Then by the first part of Theorem~\ref{thm:bfhrr}, it is enough to show that $\m L = L \cap (\m M)$ if and only if $(A :_B \m) = (0 :_B \m) + A$.  But $(A :_B \m) = (M/\m L)^\vee$, and by Lemma~\ref{lemma:dualofintersectionissum} \[
\left(\frac{M}{L \cap \m M}\right)^\vee = \left(\frac ML\right)^\vee + \left(\frac M {\m M}\right)^\vee = A + (0 :_B \m).
\]

To prove (2), we have that $A$ is \bemp\ in $B \iff$ (by Lemma~\ref{lem:bfbe}) $L$ is \bfull\ in $M \iff$ (by Theorem~\ref{thm:bfhrr}) $L/\m L$ is \bfull\ in $M/\m L \iff \left(\frac{M/\m L}{L/\m L}\right)^\vee = (M/L)^\vee=A$ is \bemp\ in $(M/\m L)^\vee = (A :_B \m)$.

To prove (3), we have by Lemma~\ref{lem:bfbe} that $L$ is \bfull\ in $M$, and hence by Theorem~\ref{thm:bfhrr} that $\len(M/L)<\infty$.  But (finite) length is preserved by Matlis duality, so we also have $\len(A) = \len((M/L)^\vee)<\infty$.

To prove (4).  We have $A = (M/L)^\vee$, so $\len(A) = \len((M/L)^\vee) = \len(M/L)<\infty$.  By Lemma~\ref{lem:bfbe} and Theorem~\ref{thm:bfhrr}, it is enough to show the equivalence: $A=\m(A:_B \m) \iff L = (\m L :_M \m)$.  But that in turn follows from Theorem \ref{thm:coldual}.

Finally we prove (5).  As in the proof of (4) above $A=(M/L)^\vee$.  Set $A_n=\m^n(A:_B\m^n)$ and $L_n=(\m^nL:_M\m^n)$. Thus by Lemma~\ref{lem:bfbe} and Theorem~\ref{thm:bfhrr}, it is enough to show the equivalence $A_n$ \bemp\ if and only if $L_n$ \bfull.  But as in (4), this follows from
Theorem \ref{thm:coldual}.
\end{proof}

As a result of Theorem \ref{thm:be}, we see that basically empty and $\m$-basically empty share the same relationship observed in Remark \ref{rem:twobasfull} between basically full and $\m$-basically full.

\begin{cor}
Let $(R,\m)$ be a Noetherian local ring, and $A \subseteq B$ Artinian $R$-modules such that $A$ is finite-length. Then $\Jintrel \m BA$ is the largest submodule of $A$ that is basically empty in $B$.
\end{cor}

\begin{proof}
This follows from part (4) of Theorem \ref{thm:be}.
\end{proof}

These results lead us to the following characterization of principal ideal rings among complete Noetherian local rings:

\begin{prop}
Let $(R,\m)$ be a complete Noetherian local ring. The following are equivalent:
\begin{enumerate}
    \item All nonzero $\m$-primary ideals are basically full.
    \item $\m$ is principal, or equivalently, $R$ is a principal ideal ring.
    \item All finite length submodules of $E$ are basically empty.
    \item $\soc E$ is principally cogenerated.
\end{enumerate}
\end{prop}

\begin{proof}
By \cite[Theorem 7.5]{HRR-bf}, (1) and (2) are equivalent. By Lemma \ref{lem:bfbe}, (1) and (3) are equivalent. By Matlis duality, (2) and (4) are equivalent.
\end{proof}

We conclude by determining a numerical criterion for the basic emptiness of a submodule $A \subseteq B$ in terms of an inequality between the minimal numbers of cogenerators of $B/C$ and $B/A$ for any submodule $C \subseteq A$ that is covered by $A$: 

\begin{defn}
\cite[Definition 2.1]{RRcover}
If $L \subseteq N \subseteq M$ are $R$-modules, we say that $N$ is a \emph{cover} of $L$ (or $L$ is \emph{covered} by $N$) in $M$ if $N/L$ is a simple $R$-module.  In particular, if $(R,\m)$ is a local ring, $N$ is a cover of $L$ in $M$ if $N/L \cong R/\m$.
\end{defn}

Covers are connected to basic fullness, as seen in the following proposition.

\begin{prop}\label{prop:fill}
\cite[Theorem 2.17]{HRR-bf}
Let $(R,\m)$ be a Noetherian local ring and $L \subseteq M$ be finite $R$-modules such that $\len(M/L)< \infty$.  Then $L$ is basically full in $M$ if and only if $\mu(N) \leq \mu(L)$ for all covers $N$ of $L$.
\end{prop}

\begin{defn}
Let $(R, \m)$ be a Noetherian local ring and $M$ a finitely cogenerated  (i.e., Artinian -- see Proposition~\ref{pr:fcoglam}(2)) $R$-module. We denote the minimal number of cogenerators of $M$ by $\eta (M)$.
\end{defn}

\begin{lemma}\label{lem:covers}
Let $(R,\m)$ be a Noetherian complete local ring and $ C \subset A \subset B$ be Artinian $R$-modules such that $\len(A)< \infty$. Let $M := B^\vee$, $L := (B/A)^\vee$, and $K:=(B/C)^{\vee}$, considered as submodules of $M$. Then $C$ is covered by $A$ if and only if $K$ is a cover of $L$.
\end{lemma}

\begin{proof}
$C$ is covered by $A$ if and only if $A/C \cong R/\m$.  As 
\[
K/L=(B/C)^\vee/(B/A)^\vee \cong (A/C)^\vee,
\]
and $(R/\m)^{\vee} \cong R/\m$, we see that $A$ is a cover of $C$ if and only if $K$ is a cover of $L$. 
\end{proof}

\begin{prop}
Let $(R,\m)$ be a Noetherian complete local ring and $A \subseteq B$ be Artinian $R$-modules such that $\len(A)< \infty$.  Then $A$ is basically empty in $B$ if and only if $\eta(B/C) \leq \eta(B/A)$ for all  $C \subseteq A$ such that $A$ is a cover of $C$.
\end{prop}

\begin{proof}
Let $M := B^\vee$ and suppose $C \subseteq A$ is an arbitrary submodule.  Set  $L := (B/A)^\vee$ and $K:=(B/C)^{\vee}$, both considered as a submodules of $M$.  

By Lemma~\ref{lem:bfbe}, $A$ is basically empty in $B$ if and only if $L$ is \bfull\ in $M$.    Applying Proposition \ref{prop:fill}, we know $L$ is \bfull\ in $M$ if and only if for every cover $K$ of $L$, $\mu(K) \leq \mu(L)$.  By Lemma \ref{lem:covers}, $K$ is a cover of $L$ if and only if $C$ is covered by $A$; hence, the criterion $\mu(K) \leq \mu(L)$ for every cover $K$  of $L$ corresponds to $\eta(B/C) \leq \eta(B/A)$ for every $C$ which is covered by $A$.
\end{proof}

\section{Core-hull duality for nonresidual closures and relative interiors}
\label{sec:nonresidualcorehullduality}

In this section we define expansions and hulls for relative interior operation and give a nonresidual version of the core-hull duality from \cite{ERGV-chdual}, thus generalizing the material recalled at the end of Section~\ref{sec:backgroundpairops} of the current paper. We will then be able to study integral hulls (Section \ref{sec:integralhulls}), as well as cores and hulls for basically full closure and basically empty interior (Section \ref{sec:examples}). 

We begin by defining $\int$-expansions and $\int$-hulls for a relative interior $\int$. For definitions relating to absolute interiors from earlier papers, the reader may consult the end of Section~\ref{sec:backgroundpairops}.

\begin{defn}
Let $R$ be an associative (not necessarily commutative) ring and $A \subseteq B$ left $R$-modules.  Let $\int$ be a relative interior operation on a class of pairs of $R$-modules $\cP$ as in Definition \ref{def:pairops}, and assume $(A,B) \in \cP$.
We say $C$ with $A \subseteq C \subseteq B$ is an $\int$-expansion of $A$ in $B$ if $(C,B) \in \cP$ and $A_{\int}^B=C_{\int}^B$.

If $(A,B) \in \cP$, the \emph{$\int$-hull of a submodule $A$ with respect to $B$} is the sum of all $\int$-expansions of $A$ in $B$, or
 \[
\int\hull^B(A):=\sum \{C \mid \int(C) \subseteq A \subseteq C \subseteq B \text{ and } (C,B) \in \cP\}.
\]
\end{defn}

The following is a generalization of \cite[Theorem 6.3]{ERGV-chdual} for relative interiors and their dual closures. 
Note that using the nonresidual closure-interior duality of Section \ref{sec:nonresidualdual} makes the proof of the following result much easier than the proof of the original result in \cite{ERGV-chdual}.

\begin{thm}[{cf. \cite[Theorem 6.3]{ERGV-chdual}}]\label{thm:expreddual}
 Let $R$ be a Noetherian complete local ring.  Let $\int$ be a relative interior operation on pairs $A \subseteq B$ of $R$-modules that are both Noetherian or both Artinian, and let $\cl := \int^\dual$ be its dual closure operation.  There exists an order reversing one-to-one correspondence between the poset of $\int$-expansions of $A$ in $B$ and the poset of $\cl$-reductions of $(B/A)^{\vee}$ in $B^{\vee}$. Under this correspondence, an $\int$-expansion $C$ of $A$ in $B$ maps to $(B/C)^\vee$, a $\cl$-reduction of $(B/A)^\vee$ in $B^\vee$.
\end{thm}

\begin{proof}
First we show that $C$ is an $\int$-expansion of $A$ in $B$ if and only if $(B/C)^\vee$ is a $\cl$-reduction of $(B/A)^\vee$ in $B^\vee$. $C$ is an $\int$-expansion of $A$ in $B$ if and only if $A \subseteq C \subseteq B$ and $A_{\int}^B = C_{\int}^B$. This occurs if and only if
\[\left(\frac{B^\vee}{\left(\left(\frac{B}{A}\right)^\vee \right)^\cl_{B^\vee} } \right)^\vee=\left(\frac{B^\vee}{\left(\left(\frac{B}{C}\right)^\vee \right)^\cl_{B^\vee} } \right)^\vee.\]
Since the modules in question are Matlis-dualizable and $(B/C)^\vee \subseteq (B/A)^\vee$, this happens if and only if
\[\left(\left(\frac{B}{A}\right)^\vee \right)^\cl_{B^\vee}=\left(\left(\frac{B}{C}\right)^\vee \right)^\cl_{B^\vee}. \]

The correspondence is order-reversing since $C \subseteq D$ if and only if $(B/D)^\vee \subseteq (B/C)^\vee$.
\end{proof}

We can also generalize \cite[Propositions 6.4 and 6.5]{ERGV-chdual} to show that for a relative interior operation $\int$, maximal $\int$-expansions exist and every $\int$-expansion of a submodule is contained in a maximal $\int$-expansion.  The proofs are quite similar to those in \cite{ERGV-chdual} making use of Proposition~\ref{NR Ext},
Proposition~\ref{pr:dualityprops}, and Theorem \ref{thm:expreddual}; hence we do not include them here.

\begin{prop}
\label{prop:maxintexpansions}
(cf. \cite[Proposition 6.4]{ERGV-chdual}) Let $(R, \m, k)$ be a complete Noetherian local ring.
Let $A \subseteq B$ be Artinian $R$-modules and $\int$ a relative Nakayama interior defined on Artinian $R$-modules.  Maximal $\int$-expansions of $A$ exist in $B$.  In fact, if $C$ is an $\int$-expansion of $A$ in $B$, then there is some maximal expansion $D$ of $A$ in $B$ such that $A \subseteq C \subseteq D \subseteq B$.
\end{prop}

\begin{prop}
\label{prop:maxexpansionsNoeth}
(cf. \cite[Proposition 6.5]{ERGV-chdual}) Let $R$ be an associative (i.e. not necessarily commutative) ring with identity.  Let $L\subseteq M$ be (left) $R$-modules, and let $\int$ be a relative interior operation on submodules of $M$.   Let $U$ be an $\int$-expansion of $L$ in $M$.  Assume $M/U$ is Noetherian.  Then there is an $R$-module $N$ with $U \subseteq N \subseteq M$, such that $N$ is a maximal $\int$-expansion of $L$ in $M$.
\end{prop}

Using the duality we developed between minimal generating sets and minimal cogenerating sets in \cite{ERGV-chdual}, we get the next result. Its proof follows by using Theorem \ref{thm:expreddual} above in place of \cite[Theorem 6.3]{ERGV-chdual}, with the additional note that \cite[Theorem 2.9]{ERGV-chdual} does not require the closure to be residual.

\begin{prop}\label{prop:maxexpmincogen}
(cf. \cite[Proposition 6.14]{ERGV-chdual}) Let $(R,\m)$ be a Noetherian local ring and $\int$ a Nakayama relative interior on Artinian $R$-modules.  Let $A \subseteq B$ be Artinian $R$-modules.  Suppose that $C \subseteq D$ are $\int$-expansions of $A$ in $B$, with $D$ a maximal $\int$-expansion.  Then any minimal cogenerating set of $B/D$ extends to a minimal cogenerating set for $B/C$.
\end{prop}

We note that Lemma~\ref{lemma:dualofintersectionissum} is needed for the duality between $\int$-hulls and $\cl$-cores for relative interiors and closure operations:

\begin{thm}[cf. Theorem 6.17 of \cite{ERGV-chdual}]
\label{thm:hullexists}
Let $R$ be a complete Noetherian local ring.  Let $A \subseteq B$ be Artinian $R$-modules, and let $\int$ be a relative Nakayama interior defined on Artinian $R$-modules. Then the $\int$-hull of $A$ in $B$ is dual to the $\cl$-core of $(B/A)^\vee$ in $B^\vee$, where $\cl$ is the closure operation dual to $\int$.
\end{thm}

\begin{proof}
The proof is identical to \cite[Theorem 6.17]{ERGV-chdual} again using Lemma~\ref{lemma:dualofintersectionissum} and Proposition \ref{pr:dualityprops}.
\end{proof}

\section{Formulas for integral-hulls for submodules of injective hull of the residue field}\label{sec:integralhulls}

The obstruction to applying the results of \cite{nmeRG-cidual} and \cite{ERGV-chdual} to integral closure
is that most methods of extending integral closure to modules do not give residual closures. Now that we know that the dual of any Nakayama closure operation $\cl$ defined on Matlis dualizable modules is a Nakayama interior $\int$ (Proposition \ref{pr:nakdual}), we can define a dual to integral closure of modules. We call this integral interior.  There are various generalizations of integral closure to modules, and the following definition applies to any of them.  In fact, understanding the integral closure of an ideal in a Noetherian ring is sufficient for the results in this section.
 
 We also include some striking results which tie our work on basically full closure/basically empty interior to the tight$\hull$/tight$\core$ of certain modules.

\begin{defn}
For $N \subseteq M$ be finitely generated modules, let $\alpha(N,M)=N_M^-$ be the integral closure of $N$ in $M$. If $A \subseteq B$ are Artinian $R$-modules, then we define $A_{-}^B:=\alpha^{\dual}(A,B)$
to be the integral interior of $A$ in $B$.  We call a submodule $A \subseteq B$ integrally open in $B$ if $A_{-}^B=A$.
\end{defn}

The following gives a criterion for determining which submodules of the injective hull of the residue field of a complete local domain $R$ are integrally open and is in fact the dual to Proposition \ref{pr:vascriterion}.

\begin{thm}\label{thm:intopenJbe}
Let $(R,\m)$ be a complete local domain and $L \subseteq E=E_R(k)$.  $L$ is integrally open in $E$ if and only if $L$ is $J$-basically empty in $E$ for all ideals $J \subseteq R$.
\end{thm}

\begin{proof}
Let $I=(0:_R L)$.  We know that $I$ is integrally closed if and only if $I=\Jcol J I R=(JI:_RJ)$ for all ideals $J \subseteq R$ by Proposition \ref{pr:vascriterion}.  Thus $I$ is integrally closed if and only if
\[
L=(R/I)^{\vee}=\left(\displaystyle\frac{R}{(JI:J)}\right)^{\vee}=J(L:_EJ)
\]
for all ideals $J \subseteq R$ by Theorem~\ref{thm:coldual}.
Note that $\Jcolsym J^\dual(L,E)=\Jintrel{J}{E}{L}=J(L:_E J)$ by Theorem~\ref{thm:Jcoldual}.  Hence $L$ is integrally open in $E$ if and only if $L$ is $J$-basically empty in $E$ for all ideals $J$ of $R$.
\end{proof}

Using the duality of $\int$-expansions/$\cl$-reductions and $\int$-hull/$\cl$-core for relative interiors $\int$ developed in Section \ref{sec:nonresidualcorehullduality}, we get a duality between the integral$\core$ of finitely generated modules and integral$\hull$ of Artinian modules.  Under certain assumptions on the ring and the ideal itself which we will detail below, there are some nice formulas for the core of an ideal, which involve colons.  In conjunction with Theorem~\ref{thm:coldual} and our core/hull duality (Theorem~\ref{thm:hullexists}),  we can determine the integral-hull of some submodules of the injective hull $E_R(k)$ of the residue field of a Noetherian local ring $(R,\m,k)$.

Before stating the known formulas for core, we recall some definitions and results for cores of ideals.

\begin{defn}\label{defn:Gs} (See \cite[Section 2]{CPU2002}.)
Let $R$ be a Noetherian ring and $I$ an ideal of $R$.  We say that $I$ satisfies the condition $G_s$ if for every prime $\fp \supseteq I$ with ${\rm dim}(R_{\fp}) \leq s-1$, the minimal number of generators of $I_\fp$ is less than or equal to $\dim R_\p$.
\end{defn}

\begin{defn}\label{defn:weakresS2} (See \cite[Section 2]{CPU2002}.)
A proper ideal $K$ of $I$ is a geometric $s$-residual intersection of $I$ if there exists an $s$-generated ideal $\mathfrak{a} \subseteq I$ such that $K=I:_R\mathfrak{a}$ and ${\rm ht}(I+K) \geq s+1$.  We say $I$ is weakly $s$-residually $S_2$ if $R/K$ is $S_2$ for every geometric $i$-residual intersection $K$ and every $i \leq s$.
\end{defn}

 Huneke and Swanson's paper \cite{HuSw-core} studying the core of ideals in a two-dimensional regular local ring discussed several computational techniques. The next push for finding a formula for the core of an ideal was pursued by Corso, Polini, and Ulrich \cite{CPU2002}.  They conjectured that when $(R,\m)$ is Cohen-Macaulay with infinite residue field, $I$ is an ideal with analytic spread $\ell$ which satisfies $G_{\ell}$ and is weakly $(\ell-1)$-residually $S_2$, and $J$ is a minimal reduction of $I$ with reduction number $r$, then
\[
{\rm core}(I)=(J^r:_RI^r)I=(J^r:_RI^r)J=(J^{r+1}:_RI^r).
\]
In this same paper, Corso, Polini, and Ulrich prove the following:

\begin{prop}\label{CPUcoreformula} \cite[Proposition 5.3]{CPU2002}
Let $(R, \m)$ be a Gorenstein ring with infinite residue field and $I$ be an ideal of height at least 2, reduction number $r$ and with analytic spread $\ell$ satisfying $G_{\ell}$ and  $\depth (R/I^j) \geq {\rm dim}(R/I)-j+1$ with $1 \leq  j \leq \ell-g+1$.  Then for any minimal reduction $J$ of $I$,
\[
{\rm core}(I)=(J^r:_RI^r)I=(J^r:_RI^r)J=(J^{r+1}:_RI^r).
\]
\end{prop}

In 2005 two improvements on Proposition~\ref{CPUcoreformula} were made by Huneke and Trung \cite{HuTrcore} and Polini and Ulrich \cite{PUcoreform}.  Huneke and Trung showed that equimultiple ideals of positive height in Cohen-Macaulay local rings of characteristic 0 satisfy one of the nice formulas for core conjectured by Corso, Polini, and Ulrich.

\begin{defn}(See \cite{Shah-equimultiple}.)
Let $R$ be a Noetherian ring and $I$ an ideal of $R$ of height $h$.  We say that $I$ is an equimultiple ideal if the analytic spread of $I$ is $h$ elements.
\end{defn}

\begin{thm}\label{thm:HuTr_core}
\cite[Theorem 3.7]{HuTrcore} Let $(R, \m)$ be a Cohen-Macaulay local ring whose residue field is characteristic $0$.  Let $I$ be an equimultiple ideal with height $g \geq 1$ and $J$ a minimal reduction of $I$ with reduction number $r$.  Then 
\[
{\rm core}(I)=(J^{r+1}:_RI^r).
\]
\end{thm}

Polini and Ulrich showed that the same formula holds for ideals of positive height in any Gorenstein local ring whose characteristic is 0 or relatively large and not just for the reduction number $r$, but any $n$ which is sufficiently large.

\begin{thm} \label{thm:PU_core}
\cite[Theorem 4.5]{PUcoreform} Let $(R, \m)$ a Gorenstein local ring with infinite residue field and $I$ an ideal with analytic spread $\ell$ and positive height $g$,  
and let $J$ be a minimal reduction of $I$ with reduction number $r$.  Assume $I$ satisfies $G_{\ell}$ and $\depth (R/I^j) \geq {\rm dim}(R/I)-j+1$ for $1 \leq j \leq \ell-g$ and either the characteristic of $R/\m$ is $0$ or $p > r-\ell+g$.  Then
\[
{\rm core}(I)=(J^{n+1}:_RI^n)
\]
for every $n \geq {\rm max}\{r-\ell+g,0\}$.
\end{thm}

Fouli, Polini, and Ulrich in \cite[Theorem 3.3]{FPUcoreform} gave another generalization of Theorem~\ref{thm:HuTr_core} that has no hypotheses on the characteristic of the residue field, but requires additional hypotheses, including that the residue field is perfect.

Using Theorem~\ref{thm:hullexists} and Theorem~\ref{thm:coldual}, we obtain formulas for some submodules of the injective hull of the residue field.  

\begin{thm} \label{thm:hullform}
If $(R,\m)$ is a complete local ring 
and $I$ is an ideal satisfying ${\rm core}(I)=(J^{n+1}:I^n)$ for some
reduction $J$ and natural number $n$,
then
\[
{\rm hull}^E(0:_E I)=I^n(0:_E J^{n+1}).
\]
\end{thm}

\begin{proof}
If ${\rm core}(I)=(J^{n+1}:I^n)$ 
then dually we have
\[
(R/({\rm core}(I)))^\vee=(R/(J^{n+1}:_R I ^n))^\vee.
\]
By the core-hull duality Theorem~\ref{thm:hullexists}, 
\[
(R/({\rm core}(I)))^\vee ={\rm hull}^E (0:_E I).
\]
Also we have
\begin{align*}
(R/(J^{n+1}:_R I ^n))^\vee=&(R/(J\cdot J^{n}:_R I ^n))^\vee  \\
&=I^n((R/J^n)^\vee:_EJ) \\
&=I^n((0:_E J^n):_EJ) \\
&=I^n(0:_E J^{n+1})\\
\end{align*}
where the second equality follows by Theorem~\ref{thm:coldual}.
Hence,
\[
{\rm hull}^E (0:_E I)=I^n(0:_E J^{n+1}).
\]
\end{proof}

Polini and Ulrich note in \cite[Remark 2.4]{PUcoreform}, that if $\text{gr}_I(R)$ is Cohen-Macaulay and $R, I$ and $J$ are as in Theorem \ref{thm:PU_core}, if $n \geq \text{max}\{r-\ell+g ,0\}$ and $i \geq 0$ then 
\[
(J^{i+n}:_RI^n)=J^i(J^n:_RI^n)=I^i(J^n:_RI^n).\] 
 In this setting, we can write ${\rm core}(I)=I(J^n:_RI^n)$ and can give a reformulation of the core of an ideal in terms of $I$-basically full interior.

\begin{prop}\label{prop:coreIbe}
Let $(R, \m)$ be a Cohen-Macualay local ring.  Suppose $R$, $I$ and $J$ satisfy the conditions in Theorem \ref{thm:PU_core}, and $\text{gr}_I(R)$ is Cohen-Macaulay,
then
\[
{\rm core} (I)=I(J^n:_RI^n)=I((J^n:_R I^{n-1}):_RI)=\Jintrel I R {(J^n:_RI^{n-1})}
\]
for $n \geq \max\{r-\ell+g+1, 1\}$.
\end{prop}

\begin{proof}
Under the hypotheses of Theorems~\ref{thm:PU_core}, \[{\rm core}(I)=(J^{n+1}:_RI^{n})=I(J^n:_R I^n)=I((J^n:_R I^{n-1}):_RI)\] for $n \geq \text{max}\{ r-\ell+g,0\}$.
By the definition of $Jbe$ interior, we see that ${\rm core} (I)=\Jintrel I R {(J^n:I^{n-1})}$. In fact, as long as $n-1 \geq \text{max}\{ r-\ell+g,0\}$, we obtain  $\Jintrel I R {(J^n:I^{n-1})}=\Jintrel I R {{\rm core}(I)}$  implying that ${\rm core}(I)$ is $I$-basically open.
\end{proof}
Dually we obtain:
\begin{thm}\label{thm:hullIbf}
Let $(R, \m)$ be a 
complete local ring  
and $I$ an ideal satisfying ${\rm core}_R(I)=I(J^n:_RI^n)$  for some 
reduction $J$ and some natural number $n \geq \text{max}\{r-\ell+g+1,1\}$, 
then
\[
{\rm hull}^E (0:_E I)=(I(0:_E (J^n:_R I^{n-1})):_E I)=\Jcol I  {(0:_E (J^n:_R I^{n-1}))} E.
\]
\end{thm}

\begin{proof}
Given
${\rm core}(I)=I(J^n:_R I^n)$, dually we have
\[
(R/{\rm core}(I))^\vee=(R/(I(J^n:_R I^n))^\vee.
\]
By the core-hull duality Theorem~\ref{thm:hullexists}, 
\[
(R/{\rm core}(I))^\vee ={\rm hull}^E (0:_E I).
\]
Also, we have
\begin{align*}
(R/I(J^n:_R I^n))^\vee &=(R/I((J^n:_R I^{n-1}):_RI))^\vee\\
&=(I(R/(J^n:_RI^{n-1}))^{\vee}:_E I)\\
&=(I(0:_E (J^n:_R I^{n-1})):_E I)
\end{align*}
where the second equality follows by Theorem~\ref{thm:coldual}.
Hence,
\[
{\rm hull}^E (0:_E I)=(I(0:_E (J^n:_RI^{n-1})):_E I)=\Jcol I {(0:_E (J^n:_R I^{n-1}))} E.
\]
\end{proof}

\begin{cor} Let $(R, \m)$ be a Cohen-Macualay local ring which satisfies the conditions in Theorem \ref{thm:hullIbf}, then ${\rm hull}^E (0:_E I)=\Jcol I {(0:_E {\rm core}(I))} E$.

\end{cor}

\begin{proof}
Since ${\rm core}(I)=(J^n:_RI^{n-1})$ for $n-1 \geq \text{max}\{r-\ell+g,0\}$ by Theorem \ref{thm:hullIbf}
we obtain ${\rm hull}^E (0:_E I)=\Jcol I {(0:_E {\rm core}(I))} E$.
\end{proof}

\subsection{Formulas for tight core/hull}

Fouli, Vassilev and Vraciu  \cite{FVV*core} devised a formula for the $*\core$ of some ideals in normal local rings of characteristic $p>0$.  They discovered three sufficient
conditions \cite[Theorems 3.7, 3.10 and 3.12]{FVV*core} such that the $*\core (I) =I(J :R I)$, giving us the
following Proposition:

\begin{prop}\label{cor:*corebe}
Let $(R, \m)$ be a normal local ring of characteristic $p>0$ with perfect residue field.  Suppose one of the following holds:
\begin{enumerate}
    \item $R$ is Cohen-Macaulay and excellent with test ideal $\tau$ and $\dim R \geq 2$.  Let $x_1, x_2,\ldots,x_d$ be part of a system of parameters and $J=(x_1^t,x_2^t,x_3, \ldots,x_d)$ where $x_1, x_2 \in \tau$ and $t \geq 3$.  Let $J \subseteq I \subseteq J^*$.
    \item The test ideal $\tau=\m$ and $J$ is any minimal $*$-reduction of $I$.
    \item The test ideal $\tau$ is $\m$-primary and $J'$ is a minimal $*$-reduction of $I'$ and $J=(J')^{[q]}$ and $I=(I')^{[q]}$ for large $q=p^e$.
\end{enumerate}
Then \[*\core (I)=I(J:_R I)=\Jintrel{I}{R}{J}.\]
\end{prop}

We similarly make use of Theorem~\ref{thm:coldual} to get a formula for the $*\hull$ of certain submodules of the injective hull of the residue field.  As with the Corollary~\ref{cor:*corebe}, we note the connection with $I$-basically full closure.

 \begin{thm}\label{thm:*hullform}
Let $(R, \m)$ be a complete normal local ring of characteristic $p>0$ 
and $I$ is an ideal satisfying $*\core_R(I)=I(J:_R I)$ for some $*$-reduction $J$.
Then \[*\hull^E(0:_E I)=\Jcol I {(0:_E J)}{E}.\]  
\end{thm}

\begin{proof}
The proof is quite similar to that of Theorem~\ref{thm:hullIbf}.  The main differences are that the closure operation is tight closure not integral closure, and the power attached to $J$ and $I$ is 1 rather than $n$.
By the core-hull duality Theorem~\ref{thm:hullexists}, 
\[
(R/({*\core}_R(I)))^\vee =*\hull^E (0:_E I).
\]
The dual of $I(J:_R I)$ is
\[
(R/(I(J:_R I)))^\vee =(I(R/J)^{\vee}:_E I)=(I(0:_E J):_E I)=\Jcol I {(0:_E J)}{E}.
\]
Hence, we similarly obtain
\[
{*\hull}^E (0:_E I)=\Jcol I {(0:_E J)}{E}. \qedhere
\]
\end{proof}

\section{Basically full core and basically empty hull, with examples}\label{sec:examples}

 By Proposition \ref{jcolnakayama}, the $J$-basically full closure ($\Jcolsym J$) is a Nakayama closure on finitely generated modules over a Noetherian local ring $R$ and similarly by Proposition \ref{pr:jbenakayama} $J$-basically empty interior ($\Jintrelsym J$) is a Nakayama interior on Artinian modules over a Noetherian local ring $R$. 
 Hence, minimal $\Jcolsym J$-reductions and maximal $\Jintrelsym J$-expansions exist and so we obtain the $\Jcolsym J$-core  as the intersection of all minimal $\Jcolsym J$-reductions and similarly the $\Jintrelsym J$-hull as the sum of all maximal $\Jintrelsym J$-expansions as developed in Section \ref{sec:nonresidualcorehullduality}.

In this section we will exhibit some examples where we determine the $\Jcolsym \m$-core and $\Jintrelsym \m$-hull of ideals in $k[[t^2,t^3]]$.
In the following examples we compute $\Jintrel{J}{R}{I}$ and $\Jcol{J}{I}{R}$ where $J=\m$ or an  ideal properly contained in $\m$.

The ideals of $R=k[[t^2,t^3]]$ were characterized in \cite[Proposition 4.1]{Va-str} and the lattice following \cite[Proposition 4.1]{Va-str} is illustrated here in Figure \ref{fig:lat}, where the boxed ideals represent $|k|$ incomparable ideals, one for each $a \in k$.

\begin{figure}[h] 
\[\xymatrix{ &R \ar@{-}[d] & \\
 &(t^2,t^3) \ar@{-}[d] \ar@{-}[dl] \\
*+[F]{ (t^2+at^3)} \ar@{-}[dr] &(t^3,t^4) \ar@{-}[d] \ar@{-}[dr]  \\
 &(t^4,t^5) \ar@{-}[d] \ar@{-}[dl] & *+[F]{ (t^3+at^4)} \ar@{-}[dl] \\
*+[F]{ (t^4+at^5)} \ar@{-}[dr] &(t^5,t^6) \ar@{-}[d] \ar@{-}[dr] \\
& \vdots\ar@{-}[d] & \vdots\ar@{-}[dl] \\
&0 & \\}\]
\caption{The lattice of ideals of $k[[t^2,t^3]]$.}
\label{fig:lat}
\end{figure}
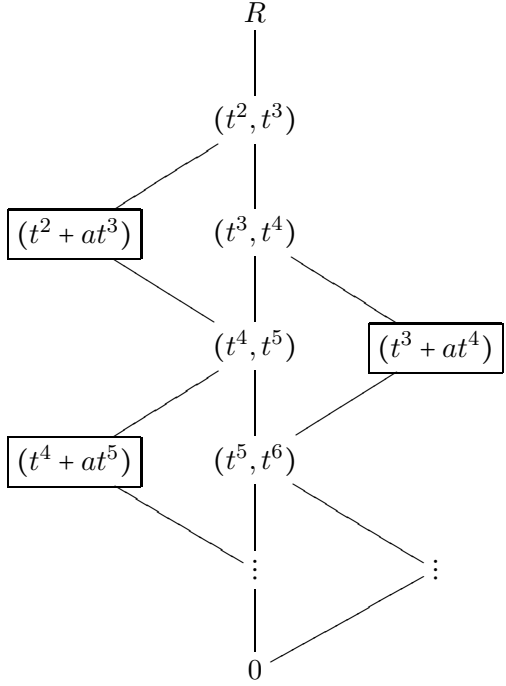
For the computations below, we will use the following lemma which illustrates the relevant colon ideals we need to compute to understand the examples.
\begin{lemma} \label{lem:coloncomp}
Let $R=k[[t^2,t^3]]$. Let $r\geq 2$ be a fixed integer and $a \in k$.  We have the following:
\begin{enumerate}
    \item $(R:_R (t^r,t^{r+1}))=R$.
    \item $(0:_R(t^r,t^{r+1}))=(0)$.
    \item $((t^n,t^{n+1}):_R (t^r, t^{r+1})) =\begin{cases} R \text{ if }  r \geq n \geq 2 \\ (t^2,t^3) \text{ if } r=n-1 \geq 2\\
(t^{n-r},t^{n-r+1}) \text{ if } 2 \leq r \leq n-2.
\end{cases}$
    \item  $((t^n+at^{n+1}):_R (t^r, t^{r+1})) =\begin{cases} R \text{ if }  r \geq n+2, n \geq 2 \\ (t^2,t^3) \text{ if } r=n+1 \geq 3\\
(t^{n-r+2},t^{n-r+3}) \text{ if } 2 \leq r \leq n.
\end{cases}$
\end{enumerate}
\end{lemma}
\begin{proof}
Recall that for any ideal $I_1 \subseteq I_2$, $(I_2:_RI_1)=R$.  The equality exhibited in (1) follows from this fact.  The equality in (2) follows as $R$ is a domain. 

For $r \geq n$, $(t^r,t^{r+1}) \subseteq (t^n,t^{n+1})$, implying the first case of (3).  However, for $2 \leq r \leq n-1$, $(t^r,t^{r+1}) \nsubseteq (t^n,t^{n+1})$.  Since $t \notin R$, when $r=n-1$, \[(t^2,t^3) = ((t^n,t^{n+1}):_R(t^{n-1},t^n))\] giving the second case in (3).  When $2 \leq r \leq n-2$, $n-r \geq n-(n-2) \geq 2$ implying that $(t^{n-r}, t^{n-r+1}) \subseteq ((t^n,t^{n+1}):_R(t^{r},t^{r+1})).$  Since $(t^{n-r-i}+at^{n-r-i+1}) \cdot t^r=t^{n-i}+at^{n-i+1} \notin (t^n,t^{n+1})$ for any $i\geq 1$ or any $a \in k$, from the lattice in Figure \ref{fig:lat}, we see that \[(t^{n-r}, t^{n-r+1}) = ((t^n,t^{n+1}):_R(t^{r},t^{r+1}))\] giving us the third case of (3).

For (4), note that the ideals $(t^n+at^{n+1})$ contain the monomials $t^s$ for $s \geq {n+2}$; however, $t^{n+1} \notin (t^n+at^{n+1})$.  
If $r \geq n+2$, then $((t^n+at^{n+1}):_R (t^r, t^{r+1}))=R$ as $(t^r,t^{r+1}) \subseteq (t^n+at^{n+1})$ giving us the first case of (4). As long as $2 \leq r\leq n$, similar to the way we argued for the third case of (3) we see that $((t^n+at^{n+1}):_R (t^r, t^{r+1}))=(t^{n-r+2},t^{n-r+3})$ giving us the third case of (4).    That leaves us with $r=n+1$ and as with the second case of (3), $((t^n+at^{n+1}):_R (t^r, t^{r+1}))=(t^2,t^3)$ giving us the second case of (4).  
\end{proof}

\begin{example}\label{ex:maxhullcore}
Let $R=k[[t^2,t^3]]$ with $\m=(t^2,t^3)$.  
We will compute  $\Jintrel{\m}{R}{I}$ and $\Jintrelsym{\m}{\hull}(I)$, followed by  $\Jcol{\m}{I}{R}$ and $\Jcolsym{\m}{\core}(I)$,   for all ideals $I$. 

From the lattice in Figure \ref{fig:lat}, we see that the nonzero nonunit ideals of $R$ are either principal of the form $(t^n + at^{n+1})$ for $a\in k$ and $n\geq 2$, or have the form $(t^n, t^{n+1})$ for $n\geq 2$. We use Lemma \ref{lem:coloncomp} to determine $\Jintrel{\m}RI$ for all ideals $I$ in $R$.  By (1) and (2) we obtain $\Jintrel{\m}{R}{R}=\m(R:_R:\m)=\m$ 
and $\Jintrel{\m}{R}0=0$. If $I=(t^n+at^{n+1})$ with $n \geq 2$ and $a\in k$ by (4) of Lemma \ref{lem:coloncomp}, 
\[\Jintrel{\m}{R}{I}=\m (I:_R\m)=(t^2,t^3)(t^n,t^{n+1})=(t^{n+2},t^{n+3})\]  
If $I=(t^n,t^{n+1})$ for $n \geq 2$, by (3) of Lemma \ref{lem:coloncomp},
\[\Jintrel{\m}{R}{I}=\m (I:_R \m)=\begin{cases}
(t^{n},t^{n+1}) \text{ if } n=2 \text{ or } n \geq 4  \\
(t^{4},t^{5}) \text{ if } n=3.\\
\end{cases}\]

By definition
\[
\Jintrelsym{\m}\hull(I)=\sum\limits_{\substack{J\supseteq I \\ \Jintrel{\m}{R}{I}=\Jintrel{\m}{R}{J}}}J.
\]

Since $\Jintrel{\m}{R}{R}=\m$,
\[\Jintrelsym{\m}\hull(R)=R \text{ and } \Jintrelsym{\m}\hull( \m )=R+\m=R.\]

The next ideals to consider in the lattice of Figure \ref{fig:lat} are $I=(t^2+at^3)$ for all $a \in k$, $I=(t^3,t^4)$ and  $I=(t^4,t^5)$.  By our computations above, we see that $\Jintrel{\m}RI=(t^4,t^5)$ for each of these possible $I$.  Thus
\[\Jintrelsym{\m}\hull(t^2+at^3)=(t^2+at^3),\ \Jintrelsym{\m}\hull(t^3,t^4)=(t^3,t^4) \text{ and }
\]
\[
\Jintrelsym{\m}\hull(t^{4},t^{5})=(t^3,t^4)+(t^4,t^5)+\sum\limits_{a \in k} (t^2+at^{3})=(t^2,t^{3}).
\]
From the lattice in Figure \ref{fig:lat}, for $n \geq 3$ and $a \in k$, 
\[\Jintrelsym{\m}\hull(t^n+at^{n+1})=(t^n+at^{n+1})
\] and for $n \geq 5$

\[
\Jintrelsym{\m}\hull(t^{n},t^{n+1})=(t^{n},t^{n+1})+\sum\limits_{a \in k} (t^{n-2}+at^{n-1})=(t^{n-2},t^{n-1}).
\]
Finally, $\Jintrelsym{\m}\hull(0)=0$ since the only ideal $I$ with $\Jintrel{\m}RI=0$ is $I=(0)$.

If $I=(t^n+at^{n+1})$ or $(t^n,t^{n+1})$, for $n \geq 2$ by (3) of Lemma \ref{lem:coloncomp}, 
\[\Jcol{\m}{I}{R}=(\m I:_R:\m)=(t^{n+2},t^{n+3}):_R(t^2,t^3)=(t^n,t^{n+1}).\]

Since
\[
{\Jcolsym{\m}}{\core}(I)=\bigcap\limits_{\substack{J\supseteq I \\ \Jcol{\m}{R}{I}=\Jcol{\m}{R}{J}}}J,
\]
we see that for $n \geq 2$,
\[
{\Jcolsym{\m}}{\core}(t^{n},t^{n+1})=(t^n,t^{n+1}) \cap \bigcap\limits_{a \in k} (t^n+at^{n+1})=(t^{n+2},t^{n+3}).
\]
 Again considering the lattice above; the only ideals $J \subseteq I$ with ideals $\Jcol{\m}{I}{R}=\Jcol{\m}{J}{R}$ are $I=R$, $I=0$  or $I=(t^n+at^{n+1})$ for $n\geq 2$, implying that  
 \[ {\Jcolsym{\m}}{\core}(I)=I\] for such ideals $I$. 
 \end{example}

 \begin{example}\label{ex:2nd}
Continuing to work in the ring $R=k[[t^2,t^3]]$, let $J=(t^r,t^{r+1})$ for any fixed $r \geq 3$. We will compute $\Jintrel{J}{R}{I}$ and $\Jintrelsym{J}{\hull}(I)$, followed by $\Jcol{J}{I}{R}$  and $\Jcolsym{J}{\core}(I)$,  for all ideals $I$.

By multiplying the ideals obtained by the colon ideals given in Lemma \ref{lem:coloncomp}, from (1) and (2) we obtain $\Jintrel{J}{R}{R}=J \cdot R=J$, $\Jintrel{J}{R}{0}=J \cdot 0=0$.  When $I=(t^n+at^{n+1})$ and $n \geq 2$,   from Lemma \ref{lem:coloncomp} (4) we obtain

\begin{equation}
\Jintrel{J}{R}{I}=J (I:_R J)=\begin{cases} (t^r,t^{r+1})R=(t^r,t^{r+1}) \text{ if } {\cvl r \geq n+2, n \geq 2} \\
(t^r,t^{r+1})(t^2,t^3)=(t^{r+2},t^{r+3}) \text{ if }{\cvl r=n+1 \geq 3}\\
(t^r,t^{r+1})(t^{n-r+2},t^{n-r+3})=(t^{n+2},t^{n+3}) \text{ if } {\cvl 2 \leq r \leq n.}\\
\end{cases}
\label{eq:prinJbe}
\tag{$\ast$}
\end{equation}

When $I=(t^n,t^{n+1})$ with $n \geq 2$, from Lemma \ref{lem:coloncomp} (3) we  obtain

\begin{equation}
\Jintrel{J}{R}{I}=J (I:_R J)=\begin{cases} (t^r,t^{r+1})R=(t^r,t^{r+1}) \text{ if } {\cvl r \geq n \geq 2} \\
(t^r,t^{r+1})(t^2,t^3)=(t^{r+2},t^{r+3}) \text{ if } {\cvl r=n-1 \geq 2}\\
(t^r,t^{r+1})(t^{n-r},t^{n-r+1})=(t^{n},t^{n+1}) \text{ if }  {\cvl 2 \leq r \leq n-2.}\\
\end{cases} 
\label{eq:2genJbe}
\tag{$\ast\ast$}
\end{equation}

We have
\[
\Jintrelsym{J}\hull(I)=\sum\limits_{\substack{J\supseteq I \\ \Jintrel{J}{R}{I}=\Jintrel{J}{R}{K}}}K, \label{eq:Jbehull}\tag{$\dagger$}
\]
as in Example \ref{ex:maxhullcore}. 
Since $(t^r,t^{r+1})=J=\Jintrel{J}{R}{R}$, we see that for any ideal $I$ with $R \supseteq I \supseteq (t^r,t^{r+1})$, then $\Jintrel{J}RR=\Jintrel{J}RI$ and for such an $I$, \[R=\Jintrelsym{J}\hull (I)= \sum\limits_{R \supseteq K\supseteq I }K.\]  By the computations above, we see that the ideals which satisfy these conditions are $I=R$,  $I=(t^n,t^{n+1})$ for $2 \leq n \leq r$ and $(t^n+at^{n+1})$, $2 \leq n\leq r-2$ when $r \geq 4$.

From the hierarchy of the lattice in Figure \ref{fig:lat}, the next ideals to consider are $I=(t^{r-1}+at^r)$ and $(t^r+at^{r+1})$ for all $a \in k$ and $(t^{r+1},t^{r+2})$.  We determine  that for all $a \in k$, \begin{align*}(t^{r+2},t^{r+3})&=\Jintrel JR{(t^{r-1}+at^r)}=\Jintrel JR{(t^{r}+at^{r+1})} \text{ by \eqref{eq:prinJbe}} \\
&=\Jintrel JR{(t^{r+1},t^{r+2})}=\Jintrel JR{(t^{r+2},t^{r+3})} \text{ by \eqref{eq:2genJbe}}.\end{align*} 
Considering the lattice in Figure \ref{fig:lat}, we can easily see that $\Jintrelsym{J}\hull(I)=I$ for $I=(t^{r-1}+at^r)$, $I=(t^r+at^{r+1})$ for all $a \in k$ and $\Jintrelsym{J}\hull(I)=(t^{r-1},t^r)$ for $I=(t^{r+1},t^{r+2})$ or $I=(t^{r+2},t^{r+3})$ by Equation \eqref{eq:Jbehull}. 

If $n-2 > r$, using \eqref{eq:Jbehull} in conjunction with \eqref{eq:prinJbe} and \eqref{eq:2genJbe}, we obtain 
\[
\Jintrelsym{J}\hull(t^{n},t^{n+1})=(t^{n-2},t^{n-1}) \text{ and } \Jintrelsym{J}\hull(t^{n}+at^{n+1})=(t^n+at^{n+1}).
\]  As in Example \ref{ex:maxhullcore}, $\Jintrelsym{J}\hull (0)=0$. 

If $r \geq 3$ and $I=(t^n+at^{n+1})$ with $a \in k$ or $I=(t^n,t^{n+1})$, for $n \geq 2$, by Lemma \ref{lem:coloncomp} (4) 
\[\Jcol{J}{I}{R}=(J I:_R:J)=(t^{n+r},t^{n+r+1}):_R(t^r,t^{r+1})=(t^n,t^{n+1}).\]
Thus considering the lattice in Figure \ref{fig:lat} we obtain,
\[
\Jcolsym{J}{\core}(t^n,t^{n+1})=\bigcap\limits_{a \in k}(t^n+at^{n+1})=(t^{n+2},t^{n+3}),
\]
and 
\[
\Jcolsym{J}{\core}(I)=I,
\]
for $I=R$, $I=(t^n+at^{n+1})$ for any $a \in k$ and $n\geq 2$ and $I=0$
by Lemma \ref{lem:coloncomp}, concluding the example.
\end{example}

\section*{Acknowledgment}
We thank Alessandra Costantini for a conversation that led to the discovery and inclusion of Theorem~\ref{thm:hullIbf}. We also wish to express our gratitude to the referees for their suggestions which greatly improved the readability of the paper.

\bibliographystyle{amsalpha}
\bibliography{BErefs}
\end{document}